\documentclass[12pt]{article}
\usepackage{latexsym, amssymb, amsthm, amsmath}
\usepackage{dsfont}

\DeclareMathAlphabet\gothic{U}{euf}{m}{n}

\makeatletter
\def\eqnarray{\stepcounter{equation}\let\@currentlabel=\theequation
\global\@eqnswtrue
\tabskip\@centering\let\\=\@eqncr
$$\halign to \displaywidth\bgroup\hfil\global\@eqcnt\z@
  $\displaystyle\tabskip\z@{##}$&\global\@eqcnt\@ne
  \hfil$\displaystyle{{}##{}}$\hfil
  &\global\@eqcnt\tw@ $\displaystyle{##}$\hfil
  \tabskip\@centering&\llap{##}\tabskip\z@\cr}

\def\endeqnarray{\@@eqncr\egroup
      \global\advance\c@equation\m@ne$$\global\@ignoretrue}

\def\@yeqncr{\@ifnextchar [{\@xeqncr}{\@xeqncr[5pt]}}
\makeatother

\allowdisplaybreaks[1]

\begin{document}

\newtheorem{lemm}{Lemma}[section]
\newtheorem{thrm}[lemm]{Theorem}
\newtheorem{coro}[lemm]{Corollary}
\newtheorem{eeg}[lemm]{Example}
\newtheorem{rrema}[lemm]{Remark}
\newtheorem{prop}[lemm]{Proposition}
\newtheorem{ddefi}[lemm]{Definition}
\newtheorem{stat}[lemm]{{\hspace{-5pt}}}

\newenvironment{rema}{\begin{rrema} \rm}{\end{rrema}}
\newenvironment{eg}{\begin{eeg} \rm}{\end{eeg}}
\newenvironment{defi}{\begin{ddefi} \rm}{\end{ddefi}}

\newcommand{\gota}{\gothic{a}}
\newcommand{\gotb}{\gothic{b}}
\newcommand{\gotc}{\gothic{c}}
\newcommand{\gote}{\gothic{e}}
\newcommand{\gotf}{\gothic{f}}
\newcommand{\gotg}{\gothic{g}}
\newcommand{\gothh}{\gothic{h}}
\newcommand{\gotk}{\gothic{k}}
\newcommand{\gotm}{\gothic{m}}
\newcommand{\gotn}{\gothic{n}}
\newcommand{\gotp}{\gothic{p}}
\newcommand{\gotq}{\gothic{q}}
\newcommand{\gotr}{\gothic{r}}
\newcommand{\gots}{\gothic{s}}
\newcommand{\gott}{\gothic{t}}
\newcommand{\gotu}{\gothic{u}}
\newcommand{\gotv}{\gothic{v}}
\newcommand{\gotw}{\gothic{w}}
\newcommand{\gotz}{\gothic{z}}
\newcommand{\gotA}{\gothic{A}}
\newcommand{\gotB}{\gothic{B}}
\newcommand{\gotG}{\gothic{G}}
\newcommand{\gotL}{\gothic{L}}
\newcommand{\gotS}{\gothic{S}}
\newcommand{\gotT}{\gothic{T}}

\newcounter{teller}
\renewcommand{\theteller}{(\alph{teller})}
\newenvironment{tabel}{\begin{list}%
{\rm  (\alph{teller})\hfill}{\usecounter{teller} \leftmargin=1.1cm
\labelwidth=1.1cm \labelsep=0cm \parsep=0cm}
                      }{\end{list}}

\newcounter{tellerr}
\renewcommand{\thetellerr}{(\roman{tellerr})}
\newenvironment{tabeleq}{\begin{list}%
{\rm  (\roman{tellerr})\hfill}{\usecounter{tellerr} \leftmargin=1.1cm
\labelwidth=1.1cm \labelsep=0cm \parsep=0cm}
                         }{\end{list}}

\newcounter{tellerrr}
\renewcommand{\thetellerrr}{(\Roman{tellerrr})}
\newenvironment{tabelR}{\begin{list}%
{\rm  (\Roman{tellerrr})\hfill}{\usecounter{tellerrr} \leftmargin=1.1cm
\labelwidth=1.1cm \labelsep=0cm \parsep=0cm}
                         }{\end{list}}

\newcounter{proofstep}
\newcommand{\nextstep}{\refstepcounter{proofstep}\vertspace \par 
          \noindent{\bf Step \theproofstep} \hspace{5pt}}
\newcommand{\firststep}{\setcounter{proofstep}{0}\nextstep}

\newcommand{\Ni}{\mathds{N}}
\newcommand{\Qi}{\mathds{Q}}
\newcommand{\Ri}{\mathds{R}}
\newcommand{\Ci}{\mathds{C}}
\newcommand{\Ti}{\mathds{T}}
\newcommand{\Zi}{\mathds{Z}}
\newcommand{\Fi}{\mathds{F}}

\renewcommand{\proofname}{{\bf Proof}}

\newcommand{\remark}{\mbox{\bf Remark} \hspace{5pt}}
\newcommand{\vertspace}{\vskip10.0pt plus 4.0pt minus 6.0pt}

\newcommand{\loc}{\mathrm{loc}}
\newcommand{\R}{\mathrm{Re} \,}
\newcommand{\I}{\mathrm{Im} \,}
\newcommand{\tr}{\mathrm{tr} \,}
\newcommand{\supp}{\mathrm{supp} \,}
\newcommand{\D}{\partial}
\newcommand{\op}{\mathrm{op}}
\newcommand{\dist}{\mathrm{dist} \,}
\newcommand{\cp}{\mathcal{P}}
\newcommand{\cn}{\mathcal{N}}
\newcommand{\cd}{\mathcal{D}}
\newcommand{\one}{\mathds{1}}

\hyphenation{groups}
\hyphenation{unitary}

\newlength{\hightcharacter}
\newlength{\widthcharacter}
\newcommand{\covsup}[1]{\settowidth{\widthcharacter}{$#1$}\addtolength{\widthcharacter}{-0.15em}\settoheight{\hightcharacter}{$#1$}\addtolength{\hightcharacter}{0.1ex}#1\raisebox{\hightcharacter}[0pt][0pt]{\makebox[0pt]{\hspace{-\widthcharacter}$\scriptstyle\circ$}}}
\newcommand{\cov}[1]{\settowidth{\widthcharacter}{$#1$}\addtolength{\widthcharacter}{-0.15em}\settoheight{\hightcharacter}{$#1$}\addtolength{\hightcharacter}{0.1ex}#1\raisebox{\hightcharacter}{\makebox[0pt]{\hspace{-\widthcharacter}$\scriptstyle\circ$}}}
\newcommand{\scov}[1]{\settowidth{\widthcharacter}{$#1$}\addtolength{\widthcharacter}{-0.15em}\settoheight{\hightcharacter}{$#1$}\addtolength{\hightcharacter}{0.1ex}#1\raisebox{0.7\hightcharacter}{\makebox[0pt]{\hspace{-\widthcharacter}$\scriptstyle\circ$}}}

\thispagestyle{empty}

\vspace*{1cm}
\begin{center}
{\Large\bf Degenerate elliptic operators  \\[3pt]
in $L_p$-spaces with complex $W^{2,\infty}$-coefficients} \\[5mm]
\large Tan Duc Do \\[10mm]

\end{center}

\vspace{5mm}

\begin{center}
{\bf Abstract}
\end{center}

\begin{list}{}{\leftmargin=1.8cm \rightmargin=1.8cm \listparindent=10mm 
   \parsep=0pt}
\item
Let $c_{kl} \in W^{2,\infty}(\Ri^d, \Ci)$ for all $k,l \in \{1, \ldots, d\}$.
We consider the divergence form operator 
$
A = - \sum_{k,l=1}^d \D_l (c_{kl} \, \D_k)
$
in $L_2(\Ri^d)$ when the coefficient matrix satisfies $(C(x) \, \xi, \xi) \in \Sigma_\theta$ for all $x \in \Ri^d$ and $\xi \in \Ci^d$, where $\Sigma_\theta$ be the sector with vertex 0 and semi-angle $\theta$ in the complex plane.
We show that for all $p$ in a suitable interval the contraction semigroup generated by $-A$ extends consistently to a contraction semigroup on $L_p(\Ri^d)$.
For those values of $p$ we present a condition on the coefficients such that the space $C_c^\infty(\Ri^d)$ of test functions is a core for the generator on $L_p(\Ri^d)$.
We also examine the operator $A$ separately in the more special Hilbert space $L_2(\Ri^d)$ setting and provide more sufficient conditions such that $C_c^\infty(\Ri^d)$ is a core.
\end{list}

\vspace{2.5cm}
\noindent
May 2016

\vspace{5mm}

\vspace{5mm}

\vspace{10mm}

\noindent
{\bf Home institution:}    \\[3mm]
Department of Mathematics \\
University of Auckland  \\ 
Private bag 92019 \\ 
Auckland 1142 \\
New Zealand  \\[5pt]
Email: tan.do@auckland.ac.nz

\newpage

\setcounter{page}{1}

\section{Introduction} \label{S4.1}

It has been known for a long time that the space of test functions $C_c^\infty(\Ri^d)$ is always a core for a strongly elliptic second-order differential operator in divergence form with Lipschitz continuous coefficients.
Nevertheless if the operator is merely degenerate elliptic, the situation is very different and it is much more difficult to prove the same type of results.
In fact $C_c^\infty(\Ri^d)$ is no longer a core in general.
Some sharp results are available in one dimension which provide characterisations for when $C_c^\infty(\Ri)$ constitutes a core, such as \cite[Theorem 3.5]{CMP}, \cite[Theorem 1.5]{DE1} and \cite[Theorem 3.3]{Do1}.
However the techniques used to prove these characterisations are intrinsically available in one dimension only.
Up to now extensions of the characterisations to higher dimensions remain widely open problems.
On the other hand, some positive results in higher dimensions are also available.
Wong-Dzung in \cite{WongDzung} proved that if a degenerate elliptic second-order differential operator in divergence form has real-valued $C^2$-coefficients, then the space $C_c^\infty(\Ri^d)$ is a core in $L_p(\Ri^d)$.
The technique used by Wong-Dzung is then refined by Ouhabaz in \cite[Theorem 5.2]{Ouh5} to prove that $C_c^\infty(\Ri^d)$ is a core for operators in $L_2(\Ri^d)$ with real-valued $W^{2,\infty}$-coefficients.
In a recent paper \cite[Propositions 4.1, 4.5, 4.6 and Theorem 4.8]{ERS5}, ter Elst, Robinson and Sikora showed the core property for the case when the coefficients are real-valued and have a mixture of smoothness between $W^{1,\infty}(\Ri^d)$ and $W^{2,\infty}(\Ri^d)$.

Apart from the interests in the core property for degenerate elliptic second-order differential operators with bounded coefficients, a large part of the literature is devoted to showing sufficient conditions under which the space of test functions is still a core for operators with real-valued coefficients which are singular either locally or at infinity.
Many interesting results can be found in \cite{Kat7}, \cite{Dav14}, \cite{Lis1}, \cite{MPPS}, \cite{MPRS}, \cite{COCCHJLY}, \cite{MR3210962} and references therein.

In this paper we investigate degenerate elliptic second-order differential operators with bounded complex-valued coefficients.
We will provide sufficient conditions for when $C_c^\infty(\Ri^d)$ is a core for these operators.
The results are generalisations of those in \cite[Theorem I]{WongDzung} and \cite[Theorem 5.2]{Ouh5}.

Let $d \in \Ni$ and $\theta \in [0, \frac{\pi}{2})$. 
Let $c_{kl} \in W^{2,\infty}(\Ri^d, \Ci)$ for all $k, l \in \{1, \ldots, d\}$.
Define $C = (c_{kl})_{1 \leq k,l \leq d}$ and $\Sigma_\theta = \{r \, e^{i \, \psi}: r \geq 0 \mbox{ and } |\psi| \leq \theta\}$.
Assume that 
\begin{equation} \label{values in sector}
(C(x) \, \xi, \xi) \in \Sigma_\theta
\end{equation}
for all $x \in \Ri^d$ and $\xi \in \Ci^d$.
Later on we will usually refer to \eqref{values in sector} as $C$ \emph{takes values in the sector $\Sigma_\theta$}.

Define the sesquilinear form
\[
\gota_0(u, v) 
= \sum_{k,l=1}^d \int_{\Ri^d} c_{kl} \, (\D_k u) \, \D_l \overline{v}
\]
on the domain $D(\gota_0) = C_c^\infty(\Ri^d)$.
Then it follows from \eqref{values in sector} that
\[
\gota_0(u, u) 
= \int_{\Ri^d} (C \, \nabla u, \nabla u)
\in \Sigma_\theta
\]
for all $u \in C_c^\infty(\Ri^d)$.
Using \cite[Theorem VI.1.27]{Kat1} we deduce that $\gota_0$ is closable.

Let $A$ be the operator associated with the closure of the form $\gota_0$.
Then $W^{2,2}(\Ri^d) \subset D(A)$ and
\[
Au = - \sum_{k,l=1}^d \D_l(c_{kl} \, \D_k u)
\]
for all $u \in W^{2,2}(\Ri^d)$.
Furthermore, by \cite[Theorem VI.2.1]{Kat1}, the operator $A$ is an $m$-sectorial operator.
Let $S$ be the $C_0$-semigroup generated by $-A$.
If $A$ is strongly elliptic, that is, if there exists a $\mu > 0$ such that 
\[
\R (C(x) \, \xi, \xi) \geq \mu \, \|\xi\|^2
\]
for all $x \in \Ri^d$ and $\xi \in \Ci^d$, then $S$ extends consistently to a $C_0$-semigroup on $L_p(\Ri^d)$ for all $p \in [1, \infty)$ by \cite[Theorem 4.8]{Aus1}.
In the general case where the coefficient matrix merely satisfies 
\[
(C(x) \, \xi, \xi) \in \Sigma_\theta
\]
for all $x \in \Ri^d$ and $\xi \in \Ci^d$, then we prove in Section \ref{S4.3} that an extension is possible for certain $p \in (1, \infty)$.
Before presenting the precise statement, we need to introduce the following notation.
We write 
\[
C = R + i \, B,
\]
where $R$ and $B$ are $d \times d$ matrix-valued functions with real-valued entries.
Let $B_a$ be the anti-symmetric part of $B$, that is, $B_a = \frac{1}{2} (B - B^T)$.
The result about semigroup extension is as follows.

\begin{prop} \label{W Lp extension}
Let $p \in (1, \infty)$. 
Suppose $|1 - \frac{2}{p}| \leq \cos \theta$ and $B_a = 0$.
Then $S$ extends consistently to a contraction $C_0$-semigroup $S^{(p)}$ on $L_p(\Ri^d)$.
\end{prop}

Let $p \in (1, \infty)$ be such that $\left| 1 - \frac{2}{p} \right| \leq \cos \theta$. 
Using Proposition \ref{W Lp extension} we can now extend $S$ consistently to a $C_0$-semigroup $S^{(p)}$ on $L_p(\Ri^d)$.
Let $-A_p$ be the generator of $S^{(p)}$.
Clearly $C_c^\infty(\Ri^d) \subset D(A_p)$.
We wish to show that $C_c^\infty(\Ri^d)$ is a core for $A_p$ under certain conditions on the coefficients.
The first main result of this paper is as follows.

\begin{thrm} \label{main theorem higher dimensions}
Let $p \in (1, \infty)$ be such that $|1 - \frac{2}{p}| < \cos\theta$.
Suppose $B_a = 0$.
Then the space $C_c^\infty(\Ri^d)$ is a core for $A_p$.
\end{thrm}

Since $A$ is naturally defined in $L_2(\Ri^d)$ via the closure of the form $\gota_0$, the condition $B_a = 0$ is not needed to obtain a $C_0$-semigroup on $L_2(\Ri^d)$.
In this case we prove that if functions in $D(A)$ are known to possess certain smoothness properties, the space $C_c^\infty(\Ri^d)$ is always a core for $A$ regardless of $B_a$.

\begin{thrm} \label{W smoothness to core}
Suppose $D(A) \subset W^{1,2}(\Ri^d)$.
Then $C_c^\infty(\Ri^d)$ is a core for $A$.
\end{thrm}

An overview of the contents of the subsequent sections is as follows.
In Section 2 we examine the matrix of coefficients $C$ closely.
Specifically we will prove various results concerning the anti-symmetric matrix $B_a$.
In Section 3 we prove the extension of the semigroup $S$ to $L_p$-spaces.
We will analyse the operator $A_p$ in detail and then prove that $C_c^\infty(\Ri^d)$ is a core for $A_p$ in Sections 4 and 5.
In Section 6 we deal specifically with the operator $A$ in $L_2(\Ri^d)$ and present the proof of Theorem \ref{W smoothness to core}.
In Section 7 we provide some interesting examples.

\section{The coefficient matrix $C$}

Define 
\[
\R C = \frac{C + C^*}{2}
\quad
{\rm and}
\quad
\I C = \frac{C - C^*}{2i},
\]
where $C^*$ is the conjugate transpose of $C$.
Then $(\R C)(x)$ and $(\I C)(x)$ are self-adjoint for all $x \in \Ri^d$ and 
\begin{equation} \label{sa form}
C = \R C + i \, \I C.
\end{equation}
We will also decompose the coefficient matrix $C$ into
\begin{equation} \label{re form}
C = R + i \, B,
\end{equation}
where $R$ and $B$ are real matrices.
Write $R = R_s + R_a$, where $R_s = \frac{R + R^T}{2}$ is the symmetric part of $R$ and $R_a = \frac{R - R^T}{2}$ is the anti-symmetric part of $R$.
Similarly $B = B_s + B_a$, where $B_s = \frac{B + B^T}{2}$ and $B_a = \frac{B - B^T}{2}$.
A comparison between \eqref{sa form} and \eqref{re form} gives
\[
\R C = R_s + i \, B_a
\quad
{\rm and}
\quad
\I C = B_s - i \, R_a.
\]
In this section we will list various relations among $R_s$, $R_a$, $B_s$ and $B_a$ which will be used in subsequent sections.

\begin{lemm} \label{W Bs < Rs}
We have
\[
|(B_s \, \xi, \eta)| 
\leq \frac{1}{2} \, \tan \theta \, \Big( (R_s \, \xi, \xi) + (R_s \, \eta, \eta) \Big)
\]
for all $\xi, \eta \in \Ri^d$.
\end{lemm}

\begin{proof}
Since $C$ takes values in $\Sigma_\theta$, we have
\[
\big|\big( (\I C(x)) \, \xi, \xi \big) \big| \leq \tan \theta \, \big( (\R C(x)) \, \xi, \xi \big)
\]
for all $x \in \Ri^d$ and $\xi \in \Ci^d$.
It follows that
\[
|(B_s \, \xi, \xi)| \leq \tan \theta \, (R_s \, \xi, \xi)
\]
for all $\xi \in \Ri^d$.
We next use polarisation to obtain
\[
|(B_s \, \xi, \eta)| 
\leq \tan \theta \, (R_s \, \xi, \xi)^{1/2} \, (R_s \, \eta, \eta)^{1/2}
\leq \frac{1}{2} \, \tan \theta \, \Big( (R_s \, \xi, \xi) + (R_s \, \eta, \eta) \Big)
\]
for all $\xi, \eta \in \Ri^d$ as required.
\end{proof}

\begin{lemm} \label{W f}
Let $j \in \{1, \ldots, d\}$.
Let $f \in W^{2,\infty}(\Ri^d)$ be such that $f(x) \geq 0$ for all $x \in \Ri^d$.
Then 
\[
|\D_j f|^2 \leq 2 \, \|\D_j^2 f\|_\infty \, f.
\]
\end{lemm}

\begin{proof}
Let $j \in \{1, \ldots, d\}$, $x \in \Ri^d$ and $h \in \Ri$.
For each $n \in \Ni$ let $f_n = J_n * f$, where $J_n$ denotes the usual mollifier with respect to a suitable function in $C_c^\infty(\Ri^d)$.
Then $f_n \geq 0$ and $f_n \in C^\infty(\Ri^d)$ for all $n \in \Ni$.
Using the Taylor expansion we have
\[
0 \leq f_n(x) + h \, (\D_j f_n)(x) + \frac{h^2}{2} \, \|\D_j^2 f_n\|_\infty
\]
for all $n \in \Ni$.
Letting $n \longrightarrow \infty$ we obtain
\[
0 \leq f(x) + h \, (\D_jf)(x) + \frac{h^2}{2} \, \|\D_j^2 f\|_\infty.
\]
This is true for all $h \in \Ri$.
Hence $|\D_j f(x)|^2 \leq 2 \, \|\D_j^2 f\|_\infty \, f(x)$ as required.
\end{proof}

\begin{lemm} \label{W f sectorial}
Let $j \in \{1, \ldots, d\}$.
Let $f \in W^{2, \infty}(\Ri^d)$ be such that $f(x) \in \Sigma_\theta$ for all $x \in \Ri^d$.
Then 
\[
|\D_j f|^2 \leq 4 \, (1 + \tan\theta)^2 \, \sup_{1 \leq j \leq d} \|\D_j^2 f\|_\infty \, \R f.
\]
\end{lemm}

\begin{proof}
Since $f(x) \in \Sigma_\theta$ for all $x \in \Ri^d$, we have $\R f \geq 0$.
Therefore by Lemma \ref{W f} we have
\[
|\D_j (\R f)|^2 
\leq 2 \, \|\D_j^2 (\R f)\|_\infty \, \R f
\leq 2 \, \sup_{1 \leq j \leq d} \|\D_j^2 f\|_\infty \, \R f.
\]
Also $|\I f| \leq (\tan\theta) \, \R f$.
That is, $(\tan\theta) \, \R f \pm \I f \geq 0$.
Applying Lemma \ref{W f} again we obtain
\begin{eqnarray*}
|\D_j ( (\tan\theta) \, \R f + \I f)|^2 
& \leq & 2 \, \|\D_j^2 ((\tan\theta) \, \R f + \I f)\|_\infty \, ((\tan\theta) \, \R f + \I f)
\\
& \leq & 2 \, (1 + \tan\theta) \, \sup_{1 \leq j \leq d} \|\D_j^2 f\|_\infty \, ((\tan\theta) \, \R f + \I f)
\end{eqnarray*}
and
\begin{eqnarray*}
|\D_j ( (\tan\theta) \, \R f - \I f)|^2 
& \leq & 2 \, \|\D_j^2 ((\tan\theta) \, \R f - \I f)\|_\infty \, ((\tan\theta) \, \R f - \I f)
\\
& \leq & 2 \, (1 + \tan\theta) \, \sup_{1 \leq j \leq d} \|\D_j^2 f\|_\infty \, ((\tan\theta) \, \R f - \I f).
\end{eqnarray*}
Adding the two inequalities gives
\[
(\tan\theta)^2 \, |\D_j (\R f)|^2 + |\D_j (\I f)|^2 
\leq 2 \, (1 + \tan\theta)^2 \, \sup_{1 \leq j \leq d} \|\D_j^2 f\|_\infty \, \R f.
\]
Hence
\[
|\D_j f|^2 = |\D_j (\R f)|^2 + |\D_j (\I f)|^2
\leq 4 \, (1 + \tan\theta)^2 \, \sup_{1 \leq j \leq d} \|\D_j^2 f\|_\infty \, \R f
\]
as required.
\end{proof}

\begin{lemm} \label{W preOleinik}
Let $j \in \{1, \ldots, d\}$.
Let $\xi, \eta \in \Ci^d$.
Then the following are valid.
\begin{tabel}
\item $|((\D_j C) \, \xi, \eta)|^2 \leq M \, \Big( ((\R C) \, \xi, \xi) + ((\R C) \, \eta, \eta) \Big)$, where 
\[
M = 8 \, (1 + \tan\theta)^2 \, (\|\xi\|^2 + \|\eta\|^2)
	\, \sup_{1 \leq l \leq d} \|\D_l^2 C\|_\infty.
\]
\item $|( (\D_j \I C) \, \xi, \eta )|^2 \leq M \, \Big( ((\R C) \, \xi, \xi) + ( (\R C) \, \eta, \eta) \Big)$, where 
\[
M = 8 \, (1 + \tan\theta)^2 \, (\|\xi\|^2 + \|\eta\|^2)
	\, \sup_{1 \leq l \leq d} \|\D_l^2 C\|_\infty.
\]
\end{tabel}
\end{lemm}

\begin{proof}
We will prove Statement (a).
The proof for Statement (b) is similar.

Since $C$ takes values in $\Sigma_\theta$, we have
\[
|(C \, \xi, \xi)| \leq (1 + \tan\theta) \, ((\R C) \, \xi, \xi).
\]
Polarisation gives
\begin{eqnarray*}
|(C \, \xi, \eta)| 
& \leq & 2 \, (1 + \tan\theta) \, ((\R C) \, \xi, \xi)^{1/2} \, 
	((\R C) \, \eta, \eta)^{1/2}
\\
& \leq & (1 + \tan\theta) \, \Big( ((\R C) \, \xi, \xi) 
	+ ((\R C) \, \eta, \eta) \Big).
\end{eqnarray*}
Let 
\[
X = (1 + \tan\theta) \, \Big( ((\R C) \, \xi, \xi) 
	+ ((\R C) \, \eta, \eta) \Big)
\]
and
\[
Y = (C \, \xi, \eta) = Y_1 + i \, Y_2,
\]
where $Y_1$ and $Y_2$ are real-valued functions.
Since $X - Y_1 \geq 0$, it follows from Lemma \ref{W f} that
\[
|\D_j (X - Y_1)|^2
\leq 2 \, \|\D_j^2(X - Y_1)\|_\infty \, (X - Y_1)
\leq 2 \, (\|\D_j^2 X\|_\infty + \|\D_j^2 Y\|_\infty) \, (X - Y_1).
\]
Arguing similarly for $X + Y_1 \geq 0$ we yield
\[
|\D_j (X + Y_1)|^2
\leq 2 \, (\|\D_j^2 X\|_\infty + \|\D_j^2 Y\|_\infty) \, (X + Y_1).
\]
By adding the two inequalities we obtain
\[
|\D_j X|^2 + |\D_j Y_1|^2 
\leq 2 \, (\|\D_j^2 X\|_\infty + \|\D_j^2 Y\|_\infty) \, X.
\]
Analogously 
\[
|\D_j X|^2 + |\D_j Y_2|^2 
\leq 2 \, (\|\D_j^2 X\|_\infty + \|\D_j^2 Y\|_\infty) \, X.
\]
Hence 
\begin{eqnarray*}
|((\D_j C) \, \xi, \eta)|^2
& = & |\D_j Y_1|^2 + |\D_j Y_2|^2
	\leq 4 \, (\|\D_j^2 X\|_\infty + \|\D_j^2 Y\|_\infty) \, X
\\
& \leq & M \, 
	\Big( ((\R C) \, \xi, \xi) + ((\R C) \, \eta, \eta) \Big),
\end{eqnarray*}
where 
\[
M = 8 \, (1 + \tan\theta)^2 \, (\|\xi\|^2 + \|\eta\|^2)
	\, \sup_{1 \leq l \leq d} \|\D_l^2 C\|_\infty.
\]
The proof is complete.
\end{proof}

Next we provide a complex version of Oleinik's inequality (cf.\ \cite{Ole}).
\begin{prop} \label{W Oleinik}
Let $j \in \{1, \ldots, d\}$.
Let $U$ be a complex $d \times d$ matrix.
Then the following are valid.
\begin{tabel}
\item $|\tr ((\D_j C) \, U)|^2 \leq M \, \Big( \tr(U^* \, (\R C) \, U) + \tr(U \, (\R C) \, U^*) \Big)$, where 
\[
M = 16 \, d \, (1 + \tan\theta)^2 \, \sup_{1 \leq l \leq d} \|\D_l^2 C\|_\infty.
\]
\item $|\tr ((\D_j \I C) \, U)|^2 \leq M \, \Big( \tr(U^* \, (\R C) \, U) + \tr(U \, (\R C) \, U^*) \Big)$, where 
\[
M = 16 \, d \, (1 + \tan\theta)^2 \, \sup_{1 \leq l \leq d} \|\D_l^2 C\|_\infty.
\]
\end{tabel}
\end{prop}

\begin{proof}
We will prove Statement (a).
The proof for Statement (b) is similar.

Let $j \in \{1, \ldots, d\}$ and 
\[
M = 16 \, d \, (1 + \tan\theta)^2 \, \sup_{1 \leq l \leq d} \|\D_l^2 C\|_\infty.
\]
Let $V$ be a unitary matrix such that $U = V \, |U|$, where $|U| = \sqrt{U^* \, U}$.
Since $|U|$ is positive and Hermitian, there exists a unitary matrix $W$ such that $|U| = W \, D \, W^*$, where $D$ is a positive diagonal matrix.
It follows that
\begin{eqnarray*}
|\tr((\D_j C) \, U)|^2
& = & |\tr((\D_j C) \, V \, |U|)|^2
	= |\tr(W^* \, (\D_j C) \, V \, W \, W^* \, |U| \, W)|^2
\\
& = & |\tr(W^* \, (\D_j C) \, V \, W \, D)|^2
	= \Big| \sum_{k=1}^d (W^* \, (\D_j C) \, V \, W)_{kk} \, D_{kk} \Big|^2
\\
& \leq & d \, \sum_{k=1}^d |(W^* \, (\D_j C) \, V \, W)_{kk} |^2 \, |D_{kk}|^2
\\
& \leq & M \, \sum_{k=1}^d \Big( (W^* \, (\R C) \, W)_{kk} 
	+ (W^* \, V^* \, (\R C) \, V \, W)_{kk} \Big) \, |D_{kk}|^2
\\
& \leq & M \, \sum_{k=1}^d \Big( D_{kk} \, (W^* \, (\R C) \, W)_{kk} \, D_{kk}
	+ D_{kk} \, (W^* \, V^* \, (\R C) \, V \, W)_{kk} \, D_{kk} \Big)
\\
& \leq & M \, \Big( \tr(|U| \, (\R C) \, |U|) + \tr(|U| \, V^* \, (\R C) \, V \, |U|) \Big)
\\
& = & M \, \Big( \tr(U \, (\R C) \, U^*) + \tr(U^* \, (\R C) \, U) \Big),
\end{eqnarray*}
where we used Lemma \ref{W preOleinik} in the second inequality.
\end{proof}

\begin{coro} \label{W Oleinik 2}
Let $j \in \{1, \ldots, d\}$.
Suppose $U$ is a complex $d \times d$ matrix with $U^T = U$.
Then the following are valid.
\begin{tabel}
\item $|\tr ((\D_j C) \, U)|^2 \leq M \, \tr(U \, R_s \, \overline{U})$, where 
\[
M = 32 \, d \, (1 + \tan\theta)^2 \, \sup_{1 \leq l \leq d} \|\D_l^2 C\|_\infty.
\]
\item $|\tr ((\D_j \I C) \, U)|^2 \leq M \, \tr(U \, R_s \, \overline{U})$, where 
\[
M = 32 \, d \, (1 + \tan\theta)^2 \, \sup_{1 \leq l \leq d} \|\D_l^2 C\|_\infty.
\]
\end{tabel}
\end{coro}

\begin{proof}
Since $U^T = U$ we have 
\begin{eqnarray*}
\tr(U^* \, (\R C) \, U) + \tr(U \, (\R C) \, U^*) 
& = & \tr(\overline{U} \, (\R C) \, U) + \tr(U \, (\R C) \, \overline{U}) 
\\
& = & \tr(\overline{U} \, (\R C) \, U) + \tr(\overline{U} \, (\R C)^T \, U)
\\
& = & 2 \, \tr(U \, R_s \, \overline{U}).
\end{eqnarray*}
Next we use Proposition \ref{W Oleinik} to derive the result.
\end{proof}

\begin{lemm} \label{W R<trR}
Let $U$ be a complex $d \times d$ matrix.
Then
\[
((\R C) \, U \, \xi, U \, \xi) \leq \tr(U^* \, (\R C) \, U) \, \|\xi\|^2
\]
for all $\xi \in \Ci^d$.
\end{lemm}

\begin{proof}
By hypothesis $\R C \geq 0$.
Therefore $((\R C) \, U \, \xi, U \, \xi) \geq 0$ for all $\xi \in \Ci^d$.
It follows that $U^* \, (\R C) \, U \geq 0$.
Hence $U^* \, (\R C) \, U \leq \tr(U^* \, (\R C) \, U) \, I$, where $I$ denotes the identity matrix.
This justifies the claim.
\end{proof}

\begin{lemm} \label{W Ba < Rs}
We have
\[
|(B_a \, \xi, \xi)| \leq (R_s \, \xi, \xi)
\]
for all $\xi \in \Ci^d$.
\end{lemm}

\begin{proof}
Write $\xi = \xi_1 + i \, \xi_2$, where $\xi_1, \xi_2 \in \Ri^d$.
Then $(R_s \, \xi, \xi) = (R_s \, \xi_1, \xi_1) + (R_s \, \xi_2, \xi_2)$ and $(B_a \, \xi, \xi) = -2i \, (B_a \, \xi_1, \xi_2)$.
Since $C$ takes values in $\Sigma_\theta$, we have $((\R C) \, \xi, \xi) \geq 0$ for all $\xi \in \Ci^d$.
Equivalently 
\[
-2 \, (B_a \, \xi_1, \xi_2) \leq (R_s \, \xi_1, \xi_1) + (R_s \, \xi_2, \xi_2).
\]
Replacing $\xi$ by $\overline{\xi}$ and repeating the same process as above we also obtain
\[
2 \, (B_a \, \xi_1, \xi_2) \leq (R_s \, \xi_1, \xi_1) + (R_s \, \xi_2, \xi_2).
\]
The result now follows.
\end{proof}

\begin{lemm} \label{W Dl Ba < Rs}
Let $l \in \{1, \ldots, d\}$ and $\xi \in \Ci^d$.
Then 
\[
|((\D_l B_a) \, \xi, \xi)|^2 \leq M \, (R_s \, \xi, \xi),
\]
where $M = 2 \, \|\xi\|^2 \, \sup_{1 \leq l \leq d} \|\D_l^2 C\|_\infty$.
\end{lemm}

\begin{proof}
Let $l \in \{1, \ldots, d\}$ and $\xi \in \Ci^d$.
By Lemma \ref{W Ba < Rs} we deduce that $R_s \pm i \, B_a \geq 0$.
Now we use Lemma \ref{W f} to derive
\begin{eqnarray*}
|(\D_l (R_s + i \, B_a) \, \xi, \xi)|^2
& \leq & 2 \, \|(\D_l^2 (R_s + i \, B_a) \, \xi, \xi)\|_\infty \, ((R_s + i \, B_a) \, \xi, \xi)
\\
& \leq & 2 \, \|\xi\|^2 \, \sup_{1 \leq l \leq d} \|\D_l^2 C\|_\infty \, ((R_s + i \, B_a) \, \xi, \xi)
\end{eqnarray*}
and
\begin{eqnarray*}
|(\D_l (R_s - i \, B_a) \, \xi, \xi)|^2
& \leq & 2 \, \|(\D_l^2 (R_s - i \, B_a) \, \xi, \xi)\|_\infty \, ((R_s - i \, B_a) \, \xi, \xi)
\\
& \leq & 2 \, \|\xi\|^2 \, \sup_{1 \leq l \leq d} \|\D_l^2 C\|_\infty \, ((R_s - i \, B_a) \, \xi, \xi).
\end{eqnarray*}
Adding the two inequalities together gives
\[
|((\D_l R_s) \, \xi, \xi)|^2 + |((\D_l B_a) \, \xi, \xi)|^2 
\leq 2 \, \|\xi\|^2 \, \sup_{1 \leq l \leq d} \|\D_l^2 C\|_\infty \, (R_s \, \xi, \xi),
\]
which clearly implies the result.
\end{proof}

\begin{lemm} \label{W Q xi square general}
Let $Q$ be a complex $d \times d$ matrix.
Suppose there exists an $M > 0$ such that $|(Q \, \xi, \xi)| \leq M \, (R_s \, \xi, \xi)$ for all $\xi \in \Ci^d$.
Then $\|Q \, \xi\|^2 \leq 4 \, M^2 \, \|R_s\|_\infty \, (R_s \, \xi, \xi)$ for all $\xi \in \Ci^d$.
\end{lemm}

\begin{proof}
Since $|(Q \, \xi, \xi)| \leq M \, (R_s \, \xi, \xi)$ for all $\xi \in \Ci^d$, polarisation gives
\[
|(Q \, \xi, \eta)| 
\leq 2 \, M \, (R_s \, \xi, \xi)^{1/2} \, (R_s \, \eta, \eta)^{1/2}
\leq 2 \, M \, \|R_s\|_\infty^{1/2} \, \|\eta\| \, (R_s \, \xi, \xi)^{1/2}
\]
for all $\xi, \eta \in \Ci^d$.
It follows that
\[
\|Q \, \xi\| \leq 2 \, M \, \|R_s\|_\infty^{1/2} \, (R_s \, \xi, \xi)^{1/2}
\]
for all $\xi \in \Ci^d$, which justifies the claim.
\end{proof}

\begin{lemm} \label{W C xi square}
We have
\[
\|C \, \xi\|^2 \leq 16 \, (1 + \tan\theta)^2 \, \, \|R_s\|_\infty \, (R_s \, \xi, \xi)
\]
for all $\xi \in \Ci^d$.
\end{lemm}

\begin{proof}
Let $\xi \in \Ci^d$.
Since $C$ takes values in $\Sigma_\theta$, we have
\[
|(C \, \xi, \xi)|
\leq ((\R C) \, \xi, \xi) + |((\I C) \, \xi, \xi)|
\leq (1 + \tan\theta) \, ((\R C) \, \xi, \xi).
\]
However $((\R C) \, \xi, \xi) \leq 2 \, (R_s \, \xi, \xi)$ by Lemma \ref{W Ba < Rs}.
It follows that 
\[
|(C \, \xi, \xi)|
\leq 2 \, (1 + \tan\theta) \, (R_s \, \xi, \xi).
\]
Using Lemma \ref{W Q xi square general} we obtain
\[
\|C \, \xi\|^2 
\leq 16 \, (1 + \tan\theta)^2 \, \|R_s\|_\infty \, (R_s \, \xi, \xi)
\]
as required.
\end{proof}

Recall that the Hilbert-Schmidt norm for a matrix $V \in M_{d \times d}(\Ci)$ is defined by 
\[
\|V\|_{HS} 
= \left( \tr (V^* \, V) \right)^{1/2} 
= \left( \sum_{j=1}^d \|V e_j\|^2 \right)^{1/2}.
\]

\begin{lemm} \label{W HS norm lemma}
Let $U$ a complex $d \times d$ matrix with $U^T = U$.
Then
\[
\|C \, U\|_{HS}^2 \leq 16 \, (1 + \tan\theta)^2 \, \|R_s\|_\infty \, \tr (U \, R_s \, \overline{U}).
\]
\end{lemm}

\begin{proof}
We note that 
\begin{eqnarray*}
\|C \, U\|_{HS}^2
= \sum_{j=1}^d \|C \, U e_j\|_2^2
& \leq & 16 \, (1 + \tan\theta)^2 \, \|R_s\|_\infty 
	\sum_{j=1}^d (R_s \, U e_j, U e_j)
\\
& = & 16 \, (1 + \tan\theta)^2 \, \|R_s\|_\infty \, \tr (U \, R_s \, \overline{U}),
\end{eqnarray*}
where we used Lemma \ref{W C xi square} in the second step.
\end{proof}

\section{$L_p$ extension} \label{S4.3}

Let $S$ be the contraction $C_0$-semigroup generated by $-A$.
In this section we will extend $S$ to a contraction $C_0$-semigroup on $L_p(\Ri^d)$ for all $p \in (1, \infty)$ with $|1 - \frac{2}{p}| \leq \cos\theta$, under the condition that $B_a = 0$.

\begin{proof}[Proof of Proposition \ref{W Lp extension}]
We proceed via two steps.
\\
{\bf Step 1:} Suppose that $A$ is strongly elliptic.
\\
Then $S$ extends consistently to a $C_0$-semigroup $S^{(p)}$ on $L_p(\Ri^d)$ by \cite[Theorem 2.21]{AMT1}.
Using duality arguments we can assume without loss of generality that $p \geq 2$.
Let $-A_p$ be the generator of $S^{(p)}$. 
Let $u \in \cd$, where $\cd = D(A) \cap D(A_p) \cap L_\infty(\Ri^d)$.
Since $A$ is strongly elliptic, the form $\gota_0$ is closable and $D(\overline{\gota_0}) = W^{1,2}(\Ri^d)$.
By construction $D(A) \subset D(\overline{\gota_0})$.
Therefore $u \in W^{1,2}(\Ri^d)$.
Set $v = |u|^{p-2} \, u$.
Then $v \in L_q(\Ri^d) \cap L_2(\Ri^d)$, where $q$ is the dual exponent of $p$.
By \cite[Lemma 7.7]{GT} we have
\[
\D_l v = \frac{p}{2} \, |u|^{p-2} \, \D_l u + \frac{p-2}{2} \, |u|^{p-4} \, u^2 \, \D_l \overline{u}
\]
for all $l \in \{1, \ldots, d\}$.
It follows that $v \in W^{1,2}(\Ri^d)$.
Our aim is to prove the inequality $\R \int (A_p u) \, \overline{v} \geq 0$, where here and in the rest of this paragraph the integral is over the set $\{x \in \Ri^d: u(x) \neq 0\}$.
Indeed we have
\begin{eqnarray*}
\int (A_p u) \, \overline{v}
& = & \int (A u) \, \overline{v}
	= \overline{\gota_0}(u, v)
	= \sum_{k,l=1}^d \int c_{kl} \, (\D_k u) \, \D_l \overline{v}
\\
& = & \sum_{k,l=1}^d \int c_{kl} \, (\D_k u) \, 
	\Big( \frac{p}{2} \, |u|^{p-2} \, \D_l \overline{u} 
	+ \frac{p-2}{2} \, |u|^{p-4} \, \overline{u}^2 \, \D_l u \Big)
\\
& = & \frac{1}{2} \int |u|^{p-4} \, \sum_{k,l=1}^d \Big( p \, c_{kl} \, |u|^2 \, (\D_k u) \, \D_l \overline{u}
	+ (p-2) \, c_{kl} \, \overline{u}^2 \, (\D_k u) \, \D_l u \Big)
\\
& = & \frac{1}{2} \int |u|^{p-4} \, 
	\Big( p \, (C \, u \, \nabla \overline{u}, u \, \nabla \overline{u})
	+ (p-2) \, (C \, \overline{u} \, \nabla u, u \, \nabla \overline{u}) \Big).
\end{eqnarray*}
Write $u \, \nabla \overline{u} = \xi + i \, \eta$, where $\xi, \eta \in \Ri^d$.
Then
\[
\R (C \, u \, \nabla \overline{u}, u \, \nabla \overline{u})
= (R_s \, \xi, \xi) + (R_s \, \eta, \eta) + 2 \, (B_a \, \xi, \eta)
= (R_s \, \xi, \xi) + (R_s \, \eta, \eta)
\]
as $B_a = 0$ by hypothesis and
\[
\R (C \, \overline{u} \, \nabla u, u \, \nabla \overline{u})
= (R_s \, \xi, \xi) - (R_s \, \eta, \eta) + 2 \, (B_s \, \xi, \eta).
\]
Therefore
\begin{eqnarray*} 
\R \int (A_p u) \, \overline{v}
& = & \int |u|^{p-4} \, \Big( (p-1) \, (R_s \, \xi, \xi) + (R_s \, \eta, \eta) 
	+ (p-2) \, (B_s \, \xi, \eta) \Big)
\\
& = & \int |u|^{p-4} \, \Big( (R_s \, \xi', \xi') + (R_s \, \eta, \eta) 
	+ \frac{p-2}{\sqrt{p-1}} \, (B_s \, \xi', \eta) \Big),
\end{eqnarray*}
where $\xi' = \sqrt{p-1} \, \xi$.
If $\theta = 0$ then it follows from Lemma \ref{W Bs < Rs} that $(B_s \, \xi', \eta) = 0$.
Consequently
\[
\R \int (A_p u) \, \overline{v}
= \int |u|^{p-4} \, \Big( (R_s \, \xi', \xi') + (R_s \, \eta, \eta) \Big)
\geq 0.
\]
If $\theta \neq 0$ then 
\[
\R \int (A_p u) \, \overline{v}
\geq \int |u|^{p-4} \, \Big( (R_s \, \xi', \xi') + (R_s \, \eta, \eta) 
	- 2 \, \cot \theta \, |(B_s \, \xi', \eta)| \Big) 
	\geq 0,
\]
where we again used Lemma \ref{W Bs < Rs} and the fact that $|1 - \frac{2}{p}| \leq \cos \theta$ is equivalent to $|p-2| \, \tan \theta \leq 2 \, \sqrt{p-1}$.
In either case the restriction $A_p|_\cd$ is accretive.
Since $\cd$ is a core for $A_p$, we also have that $A_p$ is accretive by \cite[Lemma 3.4]{LuP}.
By the Lumer-Phillips theorem, $S^{(p)}$ is a contraction semigroup.
\\
{\bf Step 2:} Suppose that $A$ is degenerate elliptic.
\\
Let $n \in \Ni$.
Let $A_{[n]} = A - \frac{1}{n} \, \Delta$, where $\Delta = \D_1^2 + \ldots + \D_d^2$.
Then $A_{[n]}$ is strongly elliptic.
Let $S^{[n]}$ be the contraction $C_0$-semigroup generated by $A_{[n]}$.
Then $S^{[n]}$ extends consistently to a contraction $C_0$-semigroup $S^{(n,p)}$ on $L_p(\Ri^d)$ by Step 1.
Using duality arguments we can assume without loss of generality that $p \in (1, 2)$.

Let $t > 0$ and $u \in L_{2,c}(\Ri^d)$.
By \cite[Corollary 3.9]{AE2} we have $\lim_{n \to \infty} S_t^{[n]} u = S_t u$ in $L_2(\Ri^d)$.
Also by \cite[Lemma 4.5]{AE2} we obtain $\lim_{n \to \infty} S_t^{[n]} u = S_t u$ in $L_1(\Ri^d)$.
Interpolation then gives $\lim_{n \to \infty} S_t^{[n]} u = S_t u$ in $L_p(\Ri^d)$.
It follows that $\|S_t u\|_p \leq \|u\|_p$ as $S^{(n,p)}$ is contractive on $L_p(\Ri^d)$.
But $L_{2,c}(\Ri^d)$ is dense in $L_2(\Ri^d) \cap L_p(\Ri^d)$.
Therefore $\|S_t u\|_p \leq \|u\|_p$ for all $u \in L_2(\Ri^d) \cap L_p(\Ri^d)$.
That is, $S_t|_{L_2(\Ri^d) \cap L_p(\Ri^d)}$ extends continuously to a contraction operator $S_t^{(p)}$ on $L_p(\Ri^d)$.
We now use \cite[Proposition 1]{Voi} to conclude that $S^{(p)}$ is a $C_0$-semigroup on $L_p(\Ri^d)$.
\end{proof}

\section{The operator $B_p$}

Let $p \in (1, \infty)$.
Let $q$ be such that $\frac{1}{p} + \frac{1}{q} = 1$.
Define 
\begin{equation} \label{W Hq}
H_q u = - \sum_{k,l=1}^d \D_k (\overline{c_{kl}} \, \D_l u)
\end{equation}
on the domain
\[
D(H_q) = C_c^\infty(\Ri^d).
\]
Next define $B_p = (H_q)^*$, which is the dual of $H_q$.
Then $B_p$ is closed by \cite[Subsection III.5.5]{Kat1}.
Also note that $W^{2,p}(\Ri^d) \subset D(B_p)$ and
\[
B_p u = - \sum_{k,l=1}^d \D_l (c_{kl} \, \D_k u)
\]
for all $u \in W^{2,p}(\Ri^d)$.

We will prove at the end of this section that $C_c^\infty(\Ri^d)$ is a core for $B_p$ if $|1 - \frac{2}{p}| < \cos\theta$ and $B_a = 0$.
In the next section we will prove that $A_p = B_p$ under the same assumptions.
The proofs require a lot of preparation.

\begin{prop} \label{W 1st ineq}
Suppose $|1 - \frac{2}{p}| \leq \cos \theta$ and $B_a = 0$.
Then
\[
\R (B_p u, |u|^{p-2} \, u) \geq 0
\]
for all $u \in W^{2,p}(\Ri^d)$.
\end{prop}

\begin{proof}
Let $u \in W^{2,p}(\Ri^d)$.
It follows from the proof of \cite[Proposition 3.5]{MS} that
\begin{eqnarray} 
(B_p u, |u|^{p-2} \, u)
& = & \int_{[u \neq 0]} |u|^{p-2} \, (C \, \nabla \overline{u}, \nabla \overline{u})
\nonumber
\\
& &
{} + (p - 2) \, \int_{[u \neq 0]} |u|^{p-4} \, 
	\big( C \, \R(u \, \nabla \overline{u}), \R(u \, \nabla \overline{u}) \big)
\nonumber
\\
& &
	{} - i \, (p - 2) \, \int_{[u \neq 0]} |u|^{p-4} \, 
	\big( C \, \R(u \, \nabla \overline{u}), \I(u \, \nabla \overline{u}) \big).
\label{W MS expansion}
\end{eqnarray}
Write $u \, \nabla \overline{u} = \xi + i \, \eta$, where $\xi, \eta \in \Ri^d$.
Then
\begin{eqnarray*}
|u|^2 \, (C \, \nabla \overline{u}, \nabla \overline{u})
& = & (C \, u \, \nabla \overline{u}, u \, \nabla \overline{u})
	= \big( C (\xi + i \, \eta), \xi + i \, \eta \big)
\\
& = & (R \, \xi, \xi) + (R \, \eta, \eta) + (B \, \xi, \eta) - (B \, \eta, \xi)
	\\
	& {} & {} - i \, \big( (R \, \eta, \xi) - (R \, \xi, \eta) + (B \, \xi, \xi) + (B \, \eta, \eta) \big).
\end{eqnarray*}
Therefore
\begin{eqnarray*}
\R \big( |u|^{2} \, (C \, \nabla \overline{u}, \nabla \overline{u}) \big)
& = & (R \, \xi, \xi) + (R \, \eta, \eta) + (B \, \xi, \eta) - (B \, \eta, \xi)
\\
& = & (R_s \, \xi, \xi) + (R_s \, \eta, \eta) + 2 \, (B_a \, \xi, \eta)
\\
& = & (R_s \, \xi, \xi) + (R_s \, \eta, \eta)
\end{eqnarray*}
since $B_a = 0$.
We also have
\[
\R \big( C \, \R(u \, \nabla \overline{u}), \R(u \, \nabla \overline{u}) \big)
= \R (C \, \xi, \xi)
= (R \, \xi, \xi)
= (R_s \, \xi , \xi).
\]
Similarly
\[
\R \big( i \, \big( C \, \R(u \, \nabla \overline{u}), \I(u \, \nabla \overline{u}) \big) \big)
= \R \big( i \, (C \, \xi, \eta) \big)
= -(B \, \xi, \eta)
= -(B_s \, \xi , \eta)
\]
since $B_a = 0$.
Hence taking the real parts on both sides of \eqref{W MS expansion} yields
\[
\R (B_p u, |u|^{p-2} \, u)
= \int_{[u \neq 0]} |u|^{p-4} \, \Big( (p-1) \, (R_s \, \xi, \xi) + (R_s \, \eta, \eta) 
	+ (p-2) \, (B_s \, \xi, \eta) \Big)
\]
since $B_a = 0$.
Now we argue as in Step 1 of the proof of Proposition \ref{W Lp extension} to derive the claim.
\end{proof}

Let $J \in C_c^\infty(\Ri^d, \Ri)$ be such that $J \geq 0$, $\supp J \subset B_1(0)$ and $\int_{\Ri^d} J = 1$.
For each $n \in \Ni$ and $x \in \Ri^d$ define $J_n(x) = n^d \, J(n \, x)$. 
For all $n \in \Ni$ define the bounded operator $T_n^{(1)}: W^{1,p}(\Ri^d) \longrightarrow L_p(\Ri^d)$ by
\[
T_n^{(1)} u 
= - \sum_{k,l=1}^d \int_{\Ri^d} J_n(y) \, 
	\Big( (I-L_y) \, (\D_l c_{kl}) \Big) \, L_y(\D_k u) \, dy
\]
and the bounded operator $T_n^{(2)}: W^{1,p}(\Ri^d) \longrightarrow L_p(\Ri^d)$ by
\[
T_n^{(2)} u
= - \sum_{k,l=1}^d \int_{\Ri^d}
	\left( \frac{\D}{\D y_l} \, \Big( J_n(y) \, (I-L_y) \, c_{kl} \Big) \right) \,
	L_y(\D_k u) \, dy,
\]
where $(L_y u)(x) = u(x - y)$ for all $x, y \in \Ri^d$.
Also define for all $n \in \Ni$ the operator $T_n: W^{1,p}(\Ri^d) \longrightarrow L_p(\Ri^d)$ by
\begin{equation} \label{W Tn definition}
T_n = T_n^{(1)} + T_n^{(2)}.
\end{equation}

\begin{lemm} \label{W Tn1}
The sequence $\{T_n^{(1)}\}_{n \in \Ni}$ is bounded.
Furthermore $\lim_{n \to \infty} \|T_n^{(1)} u\|_p = 0$ for all $u \in W^{1,p}(\Ri^d)$.
\end{lemm}

\begin{proof}
Let $n \in \Ni$ and $u \in W^{1,p}(\Ri^d)$.
For all $k,l \in \{1, \ldots, d\}$ we have $c_{kl} \in W^{2,\infty}(\Ri^d)$, which implies 
\begin{equation} \label{W c Lipschitz}
|(\D_l c_{kl})(x) - (\D_l c_{kl})(x-y)| \leq \|c_{kl}\|_{W^{2,\infty}} \, |y|
\end{equation}
for all $x, y \in \Ri^d$.
It follows that 
\begin{eqnarray*}
\|T_n^{(1)} u\|_p
& \leq & \sum_{k,l=1}^d \int_{\Ri^d} |J_n(y)| \,
	\Big\| \Big( (I-L_y) \, (\D_l c_{kl}) \Big) \,
	L_y(\D_k u) \Big\|_p \, dy
\\
& \leq & \sum_{k,l=1}^d \int_{\Ri^d} |J_n(y)| \,
	\| (I-L_y) \, (\D_l c_{kl}) \|_\infty \, 
	\| L_y (\D_k u) \|_p \, dy
\\
& \leq & \Big( \sum_{k,l=1}^d \|c_{kl}\|_{W^{2,\infty}} \Big) \, 
	\|u\|_{W^{1,p}} \, \int_{\Ri^d} |J_n(y)| \, |y| \, dy
\\
& = & \Big( \sum_{k,l=1}^d \|c_{kl}\|_{W^{2,\infty}} \Big) \, 
	\|u\|_{W^{1,p}} \, \frac{1}{n} \int_{\Ri^d} |J(y)| \, |y| \, dy,
\end{eqnarray*}
where we used $J_n(y) = n^d \, J(n \, y)$ in the last step.
Note that
\[
\lim_{n \to \infty} \frac{1}{n} \int_{\Ri^d} |J(y)| \, |y| \, dy = 0.
\]
Hence $\lim_{n \to \infty} \|T_n^{(1)} u\|_p = 0$.
Moreover, $\{T_n^{(1)}\}_{n \in \Ni}$ is bounded.
\end{proof}

\begin{lemm} \label{W Tn2}
The sequence $\{T_n^{(2)}\}_{n \in \Ni}$ is bounded.
Furthermore $\lim_{n \to \infty} \|T_n^{(2)} u\|_p = 0$ for all $u \in W^{1,p}(\Ri^d) \cap L_{p,c}(\Ri^d)$.
\end{lemm}

\begin{proof}
Let $n \in \Ni$ and $u \in W^{1,p}(\Ri^d)$.
Expanding $T_n^{(2)}$ gives
\[
T_n^{(2)} u
= - \sum_{k,l=1}^d \int_{\Ri^d}
	\Big( J_n(y) \, L_y(\D_l c_{kl}) 
	+ (\D_l J_n)(y) \, (I-L_y) \, c_{kl} \Big) \,
	L_y(\D_k u) \, dy,
\]
where we used $L_y(\D_l c_{kl}) = - \frac{\D}{\D y_l} (L_y  c_{kl})$ for all $k,l \in \{1, \ldots, d\}$.
Therefore

\pagebreak[1]

\begin{eqnarray}
\|T_n^{(2)} u\|_p
& \leq & \sum_{k,l=1}^d \left( \| (\D_l c_{kl}) \, (\D_k u) \|_p 
	+ \int_{\Ri^d} |(\D_l J_n)(y)| \, 
	\| ((I-L_y) \, c_{kl}) \, L_y(\D_k u) \|_p \, dy \right)
\nonumber
\\
& \leq & \sum_{k,l=1}^d \left( \|\D_l c_{kl}\|_\infty \, \|(\D_k u)\|_p 
	+ \int_{\Ri^d} |(\D_l J_n)(y)| \,
	\|(I-L_y) \, c_{kl}\|_\infty \, \|L_y(\D_k u)\|_p \, dy \right)
\nonumber
\\
& \leq & M \, \|u\|_{W^{1,p}},
\label{W T2 bounded}
\end{eqnarray}
where
\begin{equation} \label{M for T2}
M = \sum_{k,l=1}^d \left( \|c_{kl}\|_{W^{2,\infty}} \, 
	\Big( 1 + \int_{\Ri^d} |(\D_l J)(y)| \, |y| \, dy \Big) \right)
\end{equation}
and we used \eqref{W c Lipschitz} in the last step.
Therefore $\{T_n^{(2)}\}_{n \in \Ni}$ is bounded.

To prove to the latter statement of the lemma, we consider two cases.
\\
{\bf Case 1:} Suppose $u \in C_c^\infty(\Ri^d)$.
\\
Since $J_n$ has a compact support, we have
\[
\sum_{k,l=1}^d \int_{\Ri^d} \left( \frac{\D}{\D y_l} \, 
\Big( J_n(y) \, (I-L_y) \, c_{kl} \Big) \right) \, (\D_k u) \, dy = 0.
\]
Consequently
\begin{eqnarray*}
T_n^{(2)} u
& = & \sum_{k,l=1}^d \int_{\Ri^d} \left( \frac{\D}{\D y_l} \, 
	\Big( J_n(y) \, (I-L_y) \, c_{kl} \Big) \right) \, 
	(I-L_y) \, (\D_k u) \, dy
\\
& = & \sum_{k,l=1}^d \int_{\Ri^d}
	\Big( J_n(y) \, L_y(\D_l c_{kl}) 
	+ (\D_l J_n)(y) \, (I-L_y) \, c_{kl} \Big) \,
	(I-L_y) \, (\D_k u) \, dy.
\end{eqnarray*}
It follows that
\[
\|T_n^{(2)} u\|_p
\leq \sum_{k,l=1}^d \int_{\Ri^d} \Big(
	J_n(y) \, \|L_y (\D_l c_{kl})\|_\infty 
	+ |(\D_l J_n)(y)| \, \|(I-L_y) \, c_{kl}\|_\infty \Big) \, 
	\|(I-L_y) (\D_k u)\|_p \, dy.
\]
Note that
\begin{eqnarray*}
\|(I-L_y) (\D_k u)\|_p 
& = & \Big( \int_{\Ri^d} |(\D_k u)(x) - (\D_k u)(x-y)|^p \, dx \Big)^\frac{1}{p}
\\
& \leq & \Big( \int_{\Ri^d} (\|u\|_{W^{2,\infty}} \, |y|)^p \, 
	\one_{\supp \D_k u \, \cup \, \supp L_y(\D_k u)} \, dx \Big)^\frac{1}{p}
\\
& \leq & 2 \, |\supp \D_k u|^{1/p} \, \|u\|_{W^{2,\infty}} \, |y|
	\leq \frac{2}{n} \, |\supp u|^{1/p} \, \|u\|_{W^{2,\infty}}
\end{eqnarray*}
for all $k \in \{1, \ldots, d\}$ and $y \in \Ri^d$ such that $|y| < \frac{1}{n}$, where in the last step we used the fact that $\supp \D_k u \subset \supp u$ for all $k \in \{1, \ldots, d\}$.
Therefore
\begin{equation} \label{W T2 bounded 2}
\|T_n^{(2)} u\|_p
\leq \frac{2M}{n} \, |\supp u|^{1/p} \, \|u\|_{W^{2,\infty}},
\end{equation}
where $M$ is defined by \eqref{M for T2} and we used the fact that $\int_{\Ri^d} J_n(y) \, dy = 1$.
Hence \eqref{W T2 bounded 2} gives $\lim_{n \to \infty} \|T_n^{(2)} u\|_p = 0$.
\\
{\bf Case 2:} Suppose $u \in W^{1,p}(\Ri^d) \cap L_{p,c}(\Ri^d)$.
\\
Let $\varepsilon > 0$. 
Let $v \in C_c^\infty(\Ri^d)$ be such that $\|u-v\|_{W^{1,p}} < \frac{\varepsilon}{2M}$.
Choose an $n \in \Ni$ such that $\frac{2M}{n} \, |\supp v|^{1/p} \, \|v\|_{W^{2,\infty}} < \frac{\varepsilon}{2}$.
Then it follows from \eqref{W T2 bounded} and \eqref{W T2 bounded 2} that
\[
\|T_n^{(2)} u\|_p 
\leq \|T_n^{(2)} (u-v)\|_p + \|T_n^{(2)} v\|_p 
\leq M \, \|u-v\|_{W^{1,p}} + \frac{2M}{n} \, |\supp v|^{1/p} \, \|v\|_{W^{2,\infty}}
< \varepsilon.
\]
The proof is complete.
\end{proof}

\begin{lemm} \label{W Tn}
The sequence $\{T_n\}_{n \in \Ni}$ is bounded.
Furthermore $\lim_{n \to \infty} \|T_n u\|_p = 0$ for all $u \in W^{1,p}(\Ri^d) \cap L_{p,c}(\Ri^d)$.
\end{lemm}

\begin{proof}
This is a consequence of Lemmas \ref{W Tn1} and \ref{W Tn2}.
\end{proof}

We have the following approximation proposition (cf.\ \cite{Fri} and \cite{Kat8} for a special case of the proposition when the coefficient $c_{kl}$ are real-valued for all $k,l \in \{1, \ldots, d\}$).

\begin{prop} \label{W smoothen}
Let $u \in D(B_p) \cap W^{1,p}(\Ri^d) \cap L_{p,c}(\Ri^d)$.
Then $\lim_{n \to \infty} B_p(J_n * u) = B_p u$ in $L_p(\Ri^d)$.
\end{prop}

\begin{proof}
Let $u \in D(B_p) \cap W^{1,p}(\Ri^d) \cap L_{p,c}(\Ri^d)$.
It is well-known that $\lim_{n \to \infty} J_n*(B_p u) = B_p u$ in $L_p(\Ri^d)$.
Therefore it suffices to show that 
\[
\lim_{n \to \infty} \|B_p(J_n*u) - J_n*(B_p u)\|_p = 0.
\]
In what follows note that $L_y (\D_l u) = - \frac{\D}{\D_l} (L_y u)$ and $\D_l (J_n * u) = (\D_l J_n)* u$ for all $l \in \{1, \ldots, d\}$.
We first calculate $J_n*(B_p u)$.
Let $x \in \Ri^d$.
Define $\phi(y) = J_n(x-y)$ for all $y \in \Ri^d$.
Then $\phi \in C_c^\infty(\Ri^d)$.
By the definition of $B_p$ we have
\begin{eqnarray*}
(J_n*(B_p u))(x)
& = & \int_{\Ri^d} J_n(x-y) \, (B_p u)(y) \, dy
	= (B_p u, \phi) = (u, H_q \phi)
\\
& = & - \sum_{k,l=1}^d \int_{\Ri^d} 
	\Big( \frac{\D}{\D y_k} \big( c_{kl}(y) \, \frac{\D}{\D y_l} J_n(x-y) \big) \Big) \, u(y) \, dy
\\
& = & \sum_{k,l=1}^d \int_{\Ri^d} \Big( c_{kl}(y) \, \frac{\D}{\D y_l} J_n(x-y) \Big) \, (\D_k u)(y) \, dy
\\
& = & -\sum_{k,l=1}^d \int_{\Ri^d} (\D_l J_n)(x-y) \, (c_{kl} \, \D_k u)(y) \, dy
\\
& = & -\sum_{k,l=1}^d \int_{\Ri^d} (\D_l J_n)(y) \, (c_{kl} \, \D_k u)(x-y) \, dy
\end{eqnarray*}
for all $n \in \Ni$.
Hence
\[
J_n*(B_p u) = -\sum_{k,l=1}^d \int_{\Ri^d} (\D_l J_n)(y) \, L_y (c_{kl} \, \D_k u) \, dy
\]
for all $n \in \Ni$.

Let $n \in \Ni$.
We have
\begin{eqnarray*}
& & B_p(J_n*u) - J_n*(B_p u)
\\
& = & - \sum_{k,l=1}^d \left( \D_l \Big( c_{kl} \int_{\Ri^d} J_n(y) \, L_y (\D_k u) \, dy \Big) 
	- \int_{\Ri^d} (\D_l J_n)(y) \, L_y (c_{kl} \, \D_k u) \, dy \right)
\\
& = & - \sum_{k,l=1}^d \left( (\D_l c_{kl}) \int_{\Ri^d} J_n(y) \, L_y (\D_k u) \, dy
	+ c_{kl} \, \D_l \Big( \int_{\Ri^d} J_n(y) \, L_y(\D_k u) \, dy \Big) \right.
	\\
	& & \phantom{- \sum_{k,l=1}^d} - \int_{\Ri^d} (\D_l J_n)(y) \, L_y (c_{kl} \, \D_k u) \, dy 
\\
& = & - \sum_{k,l=1}^d \left( (\D_l c_{kl}) \int_{\Ri^d} J_n(y) \, L_y (\D_k u) \, dy
	+ c_{kl} \int_{\Ri^d} (\D_l J_n)(y) \, L_y(\D_k u) \, dy \right.
	\\
	& & \phantom{- \sum_{k,l=1}^d} - \left. \int_{\Ri^d} (\D_l J_n)(y) \, 
	L_y (c_{kl} \, \D_k u) \, dy \right).
\end{eqnarray*}
On the other hand expanding $T_n^{(1)}$ and $T_n^{(2)}$ gives
\[
T_n^{(1)} u 
= - \sum_{k,l=1}^d \int_{\Ri^d} \Big( 
	J_n(y) \, (\D_l c_{kl}) \, L_y(\D_k u)
	- J_n(y) \, L_y \big( (\D_l c_{kl}) \, \D_k u \big) \Big) \, dy
\]
and
\begin{eqnarray*}
T_n^{(2)} u 
& = & - \sum_{k,l=1}^d \int_{\Ri^d}
	\Big( J_n(y) \, L_y(\D_l c_{kl}) 
	+ (\D_l J_n)(y) \, (I-L_y) \, c_{kl} \Big) \,
	L_y(\D_k u) \, dy
\\
& = & - \sum_{k,l=1}^d \int_{\Ri^d}
	\Big( J_n(y) \, L_y \big( (\D_l c_{kl}) \, \D_k u \big) 
	+ (\D_l J_n)(y) \, c_{kl} \, L_y(\D_k u) 
	\\
	& & \phantom{- \sum_{k,l=1}^d \int_{\Ri^d}}
	- (\D_l J_n)(y) \, L_y(c_{kl} \, \D_k u) \Big) \, dy.
\end{eqnarray*}
Therefore
\begin{eqnarray*}
T_n u 
= T_n^{(1)} u + T_n^{(2)} u 
& = & - \sum_{k,l=1}^d \left( (\D_l c_{kl}) \int_{\Ri^d} J_n(y) \, L_y (\D_k u) \, dy
	+ c_{kl} \int_{\Ri^d} (\D_l J_n)(y) \, L_y(\D_k u) \, dy \right.
	\nonumber
	\\
	& & \phantom{- \sum_{k,l=1}^d} - \left. \int_{\Ri^d} (\D_l J_n)(y) \, 
	L_y (c_{kl} \, \D_k u) \, dy \right).
\end{eqnarray*}
Hence
\begin{equation} \label{W mollifying difference}
B_p(J_n*u) - J_n*(B_p u) = T_n u.
\end{equation}
The claim now follows from Lemma \ref{W Tn}.
\end{proof}

Let $\tau \in C_c^\infty(\Ri^d)$ be such that $0 \leq \tau \leq \one$, $\tau|_{B_1(0)} = 1$ and $\supp \tau \subset B_2(0)$.
Define $\tau_n(x) = \tau(n^{-1} \, x)$ for all $x \in \Ri^d$ and $n \in \Ni$.

\begin{lemm} \label{W cut-off}
Let $u \in D(B_p) \cap W^{1,p}(\Ri^d)$.
Then $\tau_n \, u \in D(B_p)$ for all $n \in \Ni$ and we have $\lim_{n \to \infty} \tau_n \, u = u$ in $D(B_p)$.
If $u$ satisfies further that $u \in W^{2,p}(\Ri^d)$ and $\nabla (B_p u) \in (L_p(\Ri^d))^d$, then $\nabla (B_p (\tau_n \, u)) \in (L_p(\Ri^d))^d$ and $\lim_{n \to \infty} \nabla (B_p (\tau_n \, u)) = \nabla (B_p u)$ in $(L_p(\Ri^d))^d$.
\end{lemm}

\begin{proof}
Let $n \in \Ni$ and $\phi \in C_c^\infty(\Ri^d)$.
Then 
\[
(\tau_n \, u, H_q \phi) = (v, \phi),
\]
where 
\begin{equation} \label{W v}
v = \tau_n \, (B_p u) + (B_p \tau_n) \, u 
	- \sum_{k,l=1}^d c_{kl} \, (\D_k u) \, (\D_l \tau_n)
	- \sum_{k,l=1}^d c_{kl} \, (\D_l u) \, (\D_k \tau_n).
\end{equation}
It follows that 
\[
\|v\|_p \leq M_1 \, \|u\|_{W^{1,p}} + \|B_p u\|_p < \infty,
\]
where $M_1 = 3 \, \sup \{\|c_{kl} \ \tau\|_{W^{2,\infty}}: 1 \leq k,l \leq d\}$.
Therefore $\tau_n \, u \in D(B_p)$ and $B_p(\tau_n \, u) = v$.

Next we consider the expression for $v$ in \eqref{W v}.
For the first term we have $\|\tau_n \, (B_p u)\|_p \leq \|B_p u\|$ for all $n \in \Ni$ and $\{\tau_n \, (B_p u)\}_{n \in \Ni}$ converges to $B_p u$ pointwise.
As a consequence $\lim_{n \to \infty} \tau_n \, (B_p u) = B_p u$ in $L_p(\Ri^d)$ by the Lebesgue dominated convergence theorem.
For the second term we notice that 
\begin{equation} \label{W tau n 1st and 2nd derivatives}
(\D_k \tau_n)(x) = \frac{1}{n} (\D_k \tau)(n^{-1} \, x)
\quad
\mbox{and}
\quad
(\D_l \D_k \tau_n)(x) = \frac{1}{n^2} (\D_l \D_k \tau)(n^{-1} \, x)
\end{equation}
for all $x \in \Ri^d$, $n \in \Ni$ and $k,l \in \{1, \ldots, d\}$.
Since $c_{kl} \in W^{2,\infty}(\Ri^d)$ for all $k,l \in \{1, \ldots, d\}$, we obtain 
\begin{equation} \label{W negligible terms}
\|(B_p \tau_n) \, u\|_p 
= \Big\| \Big( \sum_{k,l=1}^d (\D_l c_{kl}) \, (\D_k \tau_n) 
	+ c_{kl} \, (\D_l \D_k \tau_n) \Big) \, u \Big\|_p
\leq \frac{2 d^2}{n} \, \|c_{kl}\|_{W^{2,\infty}} \, 
	\|\tau\|_{W^{2,\infty}} \, \|u\|_p
\end{equation}
for all $n \in \Ni$.
It follows that $\lim_{n \to \infty} \|(B_p \tau_n) \, u\|_p = 0$.
Similarly the last two terms also converge to 0 in $L_p(\Ri^d)$.
Clearly $\lim_{n \to \infty} \tau_n \, u = u$ in $L_p(\Ri^d)$.
Hence $\lim_{n \to \infty} \tau_n \, u = u$ in $D(B_p)$.

To prove the second statement let $j \in \{1, \ldots, d\}$ and $n \in \Ni$.
Using \eqref{W v} we have
\begin{eqnarray}
\D_j (B_p(\tau_n \, u))
& = & \tau_n \, \D_j (B_p u) + (\D_j \tau_n) \, (B_p u) + (B_p \tau_n) \, (\D_j u) + (\D_j (B_p \tau_n)) \, u
\nonumber
\\
& & - \sum_{k,l=1}^d (\D_j c_{kl}) \, (\D_k u) \, (\D_l \tau_n) 
		+ c_{kl} \, (\D_j \D_k u) \, (\D_l \tau_n) + c_{kl} \, (\D_k u) \, (\D_j \D_l \tau_n)
\nonumber
\\
& & - \sum_{k,l=1}^d (\D_j c_{kl}) \, (\D_l u) \, (\D_k \tau_n)
		+ c_{kl} \, (\D_j \D_l u) \, (\D_k \tau_n) + c_{kl} \, (\D_l u) \, (\D_j \D_k \tau_n).
\hspace*{5mm}
\label{W DjBptaunu expression}
\end{eqnarray}
It follows that
\[
\|\D_j (B_p(\tau_n \, u))\|_p
\leq M_2 \, \|u\|_{W^{2,p}} + (1 \wedge \|\tau\|_{W^{1,\infty}}) \, \|B_p u\|_{W^{1,p}} < \infty,
\]
where $M_2 = 8 \, \sup \{\|c_{kl}\|_{W^{2,\infty}} \, \|\tau\|_{W^{3,\infty}} : 1 \leq k,l \leq d\}$.
Therefore $\D_j (B_p(\tau_n \, u)) \in L_p(\Ri^d)$.
Furthermore notice that 
\begin{equation} \label{W tau n 3rd derivatives}
(\D_j \D_l \D_k \tau_n)(x) = \frac{1}{n^3} (\D_j \D_l \D_k \tau)(n^{-1} \, x)
\end{equation}
for all $x \in \Ri^d$ and $k,l \in \{1, \ldots, d\}$.
Using \eqref{W tau n 1st and 2nd derivatives}, \eqref{W tau n 3rd derivatives} and repeating the arguments used in \eqref{W negligible terms} we see that all terms in the expression for $\D_j (B_p(\tau_n \, u))$ in \eqref{W DjBptaunu expression} converge to 0 in $L_p(\Ri^d)$ as $n$ tends to infinity except for the first one, whereas the first term converges to $\D_j (B_p u)$ in $L_p(\Ri^d)$ as $n$ tends to infinity.
Hence $\lim_{n \to \infty} \D_j (B_p(\tau_n \, u)) = \D_j (B_p u)$ in $L_p(\Ri^d)$.
This completes the proof.
\end{proof}

\begin{prop} \label{W Cc dense 1}
The space $C_c^\infty(\Ri^d)$ is dense in $(D(B_p) \cap W^{1,p}(\Ri^d), \|\cdot\|_{D(B_p)})$.
\end{prop}

\begin{proof}
Let $u \in D(B_p) \cap W^{1,p}(\Ri^d)$ and $\varepsilon > 0$.
For all $n \in \Ni$ set $u_n = \tau_n \, u \in D(B_p) \cap W^{1,p}(\Ri^d) \cap L_{p,c}(\Ri^d)$.
By Lemma \ref{W cut-off} we can choose an $n \in \Ni$ such that $\|u - u_n\|_{D(B_p)} < \frac{\varepsilon}{2}$.
Next for all $m \in \Ni$ set $v_m = J_m * (\tau_n \, u) \in C_c^\infty(\Ri^d)$.
We now use Lemma \ref{W smoothen} to choose an $m \in \Ni$ such that $\|u_n - v_m\|_{D(B_p)} < \frac{\varepsilon}{2}$.
Then
\[
\|u - v_m\|_{D(B_p)}
\leq \|u - u_n\|_{D(B_p)} 
	+ \|u_n - v_m\|_{D(B_p)} < \varepsilon.
\]
This verifies the claim.
\end{proof}

\begin{prop} \label{W 2nd ineq}
Suppose $|1 - \frac{2}{p}| < \cos\theta$ and $B_a = 0$.
Then there exists an $M > 0$ such that 
\[
\R (\nabla (B_p u), |\nabla u|^{p-2} \, \nabla u) \geq - M \, \|\nabla u\|_p^p
\]
for all $u \in W^{2,p}(\Ri^d)$ such that $\nabla(B_p u) \in (L_p(\Ri^d))^d$.
\end{prop}

\begin{proof}
The condition $|1 - \frac{2}{p}| < \cos\theta$ is equivalent to $|p-2| \, \tan\theta < 2 \, \sqrt{p-1}$.
Let $\varepsilon_0 \in (0, 1 \wedge (p-1))$ be such that 
\[
|p-2| \, \tan\theta \leq 2 \, \sqrt{(1-\varepsilon) \, (p-1-\varepsilon)}
\]
for all $\varepsilon \in (0, \varepsilon_0)$.
Let $\varepsilon \in (0, \varepsilon_0)$ be such that 
\begin{equation} \label{W correct varepsilon}
\varepsilon 
< \frac{\varepsilon_0}{32 \, d \, (1 + \tan\theta)^2 \, \sup_{1 \leq l \leq d} \|\D_l^2 C\|_\infty}.
\end{equation}

Let $u \in W^{2,p}(\Ri^d)$.
By Lemma \ref{W cut-off} we can assume without loss of generality that $u$ has a compact support.
For the rest of the proof, all integrations are over the set $\{x \in \Ri^d: |(\nabla u)(x)| \neq 0 \}$.
We have
\begin{eqnarray*}
(\nabla (B_p u), |\nabla u|^{p-2} \, \nabla u)
& = & - \sum_{k,l,j=1}^d \int \Big( \D_j \D_l (c_{kl} \, \D_k u) \Big) \, 
	|\nabla u|^{p-2} \, \D_j \overline{u}
\\
& = & - \sum_{k,l,j=1}^d \int \Big( \D_l \big( (\D_j c_{kl}) \, (\D_k u) 
	+ c_{kl} \, (\D_j \D_k u) \big) \Big) \,
	|\nabla u|^{p-2} \, \D_j \overline{u}
\\
& = & - \sum_{k,l,j=1}^d \int \Big( \D_l \big( (\D_j c_{kl}) \, (\D_k u) \big) \Big) \, 
	|\nabla u|^{p-2} \, \D_j \overline{u}
\\*
& & {} + \sum_{k,l,j=1}^d \int c_{kl} \, (\D_j \D_k u) \, 
	\D_l \big( |\nabla u|^{p-2} \, \D_j \overline{u} \big)
\\
& = & ({\rm I}) + ({\rm II}).
\end{eqnarray*}
We first consider the real part of (I).
We have 
\begin{eqnarray*}
-\R \sum_{k,l,j=1}^d \int \Big( \D_l \big( (\D_j c_{kl}) \, (\D_k u) \big) \Big) \, 
	|\nabla u|^{p-2} \, \D_j \overline{u}
& = & - \R \sum_{k,l,j=1}^d \int (\D_l \D_j c_{kl}) \, (\D_k u) \, 
	(\D_j \overline{u}) \, |\nabla u|^{p-2}
\\
& & {} - \R \sum_{k,l,j=1}^d \int (\D_j c_{kl}) \, (\D_l \D_k u) \, 
	(\D_j \overline{u}) \, |\nabla u|^{p-2}
\\
& = & ({\rm Ia}) + ({\rm Ib}).
\end{eqnarray*}
For (Ia) we have
\[
({\rm Ia}) 
\geq - \frac{1}{2} \sum_{k,l,j=1}^d \|c_{kl}\|_{W^{2,\infty}} 
	\int (|\D_k u|^2 + |\D_j u|^2) \, |\nabla u|^{p-2}
\geq - M_1 \, \|\nabla u\|_p^p,
\]
where $M_1 = d^2 \, \sup \{\|c_{kl}\|_{W^{2,\infty}}: 1 \leq k,l \leq d\}$.
Let $U = (\D_l \D_k u)_{1 \leq k,l \leq d}$.
For (Ib) we estimate
\begin{eqnarray*}
({\rm Ib})  
& = & - \R \sum_{j=1}^d \int \tr((\D_j C) \, U) \, (\D_j \overline{u}) \, |\nabla u|^{p-2}
\\
& \geq & - \sum_{j=1}^d \int \Big( \varepsilon \, |\tr((\D_j C) \, U)|^2 \, |\nabla u|^{p-2}
	+ \frac{1}{4 \varepsilon} \, |\D_j \overline{u}|^2 \, |\nabla u|^{p-2} \Big)
\\
& \geq & - \varepsilon' \int \tr(U \, R_s \, \overline{U}) \, |\nabla u|^{p-2}
	- M_2 \, \|\nabla u\|_p^p,
\end{eqnarray*}
where we used Corollary \ref{W Oleinik 2}(a) in the last step with $\varepsilon' = 32 \, \varepsilon \, d \, (1 + \tan\theta)^2 \, \sup_{1 \leq l \leq d} \|\D_l^2 C\|_\infty$ and $M_2 = \frac{1}{4 \varepsilon}$.
Note that $\varepsilon' \in (0, \varepsilon_0)$ by \eqref{W correct varepsilon}.

Next we consider the real part of (II).
Note that 
\begin{eqnarray*}
\R \sum_{k,l,j=1}^d \int c_{kl} \, (\D_j \D_k u) \, 
	\D_l \big( |\nabla u|^{p-2} \, \D_j \overline{u} \big)
& = & \R \sum_{k,l,j=1}^d \int c_{kl} \, (\D_j \D_k u) \, (\D_l \D_j \overline{u}) \, |\nabla u|^{p-2}
\\
& & {} + \R \sum_{k,l,j=1}^d \int c_{kl} \, (\D_j \D_k u) \, 
	(\D_j \overline{u}) \, \D_l(|\nabla u|^{p-2})
\\
& = & ({\rm IIa}) + ({\rm IIb}).
\end{eqnarray*}
For (IIa) we have 
\[
({\rm IIa}) 
= \int \tr(\overline{U} \, (\R C) \, U) \, |\nabla u|^{p-2}
= \int \tr(U \, R_s \, \overline{U}) \, |\nabla u|^{p-2}
\]
as $B_a = 0$.
For (IIb) we have the following estimate
\begin{eqnarray*}
({\rm IIb}) 
& = & \R \sum_{k,l,i,j=1}^d \frac{p-2}{2} \int c_{kl} \, (\D_j \D_k u) \, (\D_j \overline{u}) \, 
	\Big( (\D_l \D_i u) \, (\D_i \overline{u}) + (\D_l \D_i \overline{u}) \, (\D_i u) \Big) \, |\nabla u|^{p-4}
\\
& = & \frac{p-2}{2} \int \R \Big( \big( C \, U \, \nabla \overline{u}, \overline{U \, \nabla \overline{u}} \big)
	+ \big( C \, U \, \nabla \overline{u}, U \, \nabla \overline{u} \big) \Big) \, |\nabla u|^{p-4}
\\
& = & (p-2) \int \Big( (R_s \, \xi, \xi) - (B_s \, \xi, \eta) \Big) \, |\nabla u|^{p-4},
\end{eqnarray*}
where $\xi, \eta \in \Ri^d$ and $U \, \nabla \overline{u} = \xi + i \, \eta$.

In total we obtain
\begin{eqnarray*}
\R (\nabla (B_p u), |\nabla u|^{p-2} \, \nabla u)
& \geq & - (M_1 + M_2) \, \|\nabla u\|_p^p
	+ (1-\varepsilon') \int \tr(U \, R_s \, \overline{U}) \, |\nabla u|^{p-2}
\\
& & 
	{} + (p-2) \int \Big( (R_s \, \xi, \xi) - (B_s \, \xi, \eta) \Big) \, |\nabla u|^{p-4}
\\
& = & - (M_1 + M_2) \, \|\nabla u\|_p^p + P,
\end{eqnarray*}
where 
\[
P = (1-\varepsilon') \int \tr(U \, R_s \, \overline{U}) \, |\nabla u|^{p-2}
	+ (p-2) \int \Big( (R_s \, \xi, \xi) - (B_s \, \xi, \eta) \Big) \, |\nabla u|^{p-4}.
\]
Next we will show that $P \geq 0$.
Since $B_a = 0$, it follows from Lemma \ref{W R<trR} that
\begin{eqnarray*}
(R_s \, \xi, \xi) + (R_s \, \eta, \eta) 
& = & ((\R C) \, U \, \nabla \overline{u}, U \, \nabla \overline{u})
	\leq \tr(U^* \, (\R C) \, U) \, |\nabla u|^2
\\
& = & \tr(\overline{U} \, R_s \, U) \, |\nabla u|^2
	= \tr(U \, R_s \, \overline{U}) \, |\nabla u|^2.
\end{eqnarray*}
Therefore 
\begin{eqnarray}
P 
& \geq & \int \Big( (p-1-\varepsilon') \, (R_s \, \xi, \xi) + (1-\varepsilon') \, (R_s \, \eta, \eta) 
	- (p-2) \, (B_s \, \xi, \eta) \Big) \, |\nabla u|^{p-4}
\nonumber
\\
& = & \int \Big( (R_s \, \xi', \xi') + (R_s \, \eta', \eta') 
	- \frac{p-2}{\sqrt{(1-\varepsilon') \, (p-1-\varepsilon')}} \, 
	(B_s \, \xi', \eta') \Big) \, |\nabla u|^{p-4},
\label{W P}
\end{eqnarray}
where $\xi' = \sqrt{p-1-\varepsilon'} \, \xi$ and $\eta' = \sqrt{1-\varepsilon'} \, \eta$.
If $\theta = 0$ then it follows from Lemma \ref{W Bs < Rs} that $(B_s \, \xi', \eta') = 0$.
Therefore \eqref{W P} gives
\[
P \geq \int \Big( (R_s \, \xi', \xi') + (R_s \, \eta', \eta') \Big) \, |\nabla u|^{p-4} \geq 0.
\]
If $\theta \neq 0$ then \eqref{W P} can be estimated by
\[
P \geq \int \Big( (R_s \, \xi', \xi') + (R_s \, \eta', \eta') 
	- 2 \, \cot\theta \, |(B_s \, \xi', \eta')| \Big) \, |\nabla u|^{p-4}
	\geq 0,
\]
where we again used Lemma \ref{W Bs < Rs}.
Either way we always have
\[
\R (\nabla (B_p u), |\nabla u|^{p-2} \, \nabla u) \geq - (M_1 + M_2) \, \|\nabla u\|_p^p
\]
as claimed.
\end{proof}

\begin{prop} \label{W Cc dense 2}
Suppose $|1 - \frac{2}{p}| < \cos\theta$ and $B_a = 0$.
Then $B_p$ is $m$-accretive.
Furthermore $C_c^\infty(\Ri^d)$ is a core for $B_p$.
\end{prop}

\begin{proof}
We will proceed in three steps.
\\
{\bf Step 1:} We will show that $\overline{B_p|_{D(B_p) \cap W^{1,p}(\Ri^d)}}$ is $m$-accretive.
\\
It follows from Propositions \ref{W 1st ineq} and \ref{W Cc dense 1} that $B_p|_{D(B_p) \cap W^{1,p}(\Ri^d)}$ is accretive.
Hence $\overline{B_p|_{D(B_p) \cap W^{1,p}(\Ri^d)}}$ is also accretive.

Next we will show that there exists a $\lambda > 0$ such that $(\lambda + B_p)(D(B_p) \cap W^{1,p}(\Ri^d))$ is dense in $L_p(\Ri^d)$.
In fact we will show that there exists a $\lambda > 0$ such that $W^{1,p}(\Ri^d) \subset (\lambda + B_p)(D(B_p) \cap W^{1,p}(\Ri^d))$.
Since $-\Delta$ satisfies the same conditions as those of $B_p$, Proposition \ref{W 2nd ineq} also applies to $-\Delta$.
In particular there exists an $M' > 0$ such that 
\[
\R \Big( \nabla(\Delta u), \, |\nabla u|^{p-2} \, \nabla u \Big) 
\geq - M' \, \|\nabla u\|_p^p
\]
for all $u \in W^{3,p}(\Ri^d)$.

For all $n \in \Ni$ define the operator $B_{p,n}$ by
\[
B_{p,n} u  = B_p u - \frac{1}{n} \, \Delta u
\]
on the domain
\[
D(B_{p,n}) = W^{2,p}(\Ri^d),
\]
where $\Delta = \D_1^2 + \ldots + \D_d^2$.
Note that for each $n \in \Ni$ the operator $B_{p,n}$ is strongly elliptic, which implies that $B_{p,n}$ is closed.

Let $M$ be as in Proposition \ref{W 2nd ineq} and $\lambda =  M + M' + 1$. 
Let $f \in W^{1,p}(\Ri^d)$.
Let $n \in \Ni$.
Then there exists a $u_n \in W^{2,p}(\Ri^d)$ such that $(\lambda + B_{p,n}) u_n = f$.
Elliptic regularity gives $u_n \in W^{3,p}(\Ri^d)$.
It follows that $\nabla(B_{p,n} u_n) = \nabla (f - \lambda \, u_n) \in (L_p(\Ri^d))^d$ and $\nabla (B_p u_n) = \nabla (B_{p,n} u_n) + \frac{1}{n} \, \nabla (\Delta u_n) \in (L_p(\Ri^d))^d$.
By Proposition \ref{W 1st ineq} we have
\[
(f, |u_n|^{p-2} \, u_n)
= \lambda \, \|u_n\|_p^p + (B_{p,n} u_n, |u_n|^{p-2} \, u_n) 
\geq \lambda \, \|u_n\|_p^p \geq \|u_n\|_p^p.
\]
However
\[
(f, |u_n|^{p-2} \, u_n) 
\leq \|f\|_p \, \||u_n|^{p-2} \, u_n\|_q
= \|f\|_p \, \|u_n\|_p^{p/q}
\]
by H\"{o}lder's inequality.
Therefore $\|u_n\|_p^p \leq \|f\|_p \, \|u_n\|_p^{p/q}$, or equivalently $\|u_n\|_p \leq \|f\|_p$.
Also it follows from Proposition \ref{W 2nd ineq} that
\begin{eqnarray*}
(\nabla f, |\nabla u_n|^{p-2} \, \nabla u_n)
& = & \lambda \, \|\nabla u_n\|_p^p + \R \Big( \nabla (B_{p,n} u_n), |\nabla u_n|^{p-2} \, \nabla u_n \Big)
\\
& = & \lambda \, \|\nabla u_n\|_p^p
	+ \R \Big( \nabla (B_p u_n), |\nabla u_n|^{p-2} \, \nabla u_n \Big)
	\\*
	& &
	{} - \frac{1}{n} \, \R \Big( \nabla(\Delta u_n), \, |\nabla u_n|^{p-2} \, \nabla u_n \Big)
\\*
& \geq & (\lambda - M - M') \, \|\nabla u_n\|_p^p = \|\nabla u_n\|_p^p.
\end{eqnarray*}
Again the H\"{o}lder's inequality gives $\|\nabla u_n\|_p \leq \|\nabla f\|_p$.
Hence $\|u_n\|_{W^{1,p}} \leq \|f\|_{W^{1,p}}$.
In particular $\{u_k\}_{k \in \Ni}$ is bounded in $W^{1,p}(\Ri^d)$.
Passing to a subsequence if necessary we may assume that $\{u_k\}_{k \in \Ni}$ converges weakly to a $u \in W^{1,p}(\Ri^d)$.
Note that $B_{p,n} u_n = f - \lambda \, u_n$.
Therefore $\{B_{p,n} u_n\}_{k \in \Ni}$ is bounded in $L_p(\Ri^d)$.
Passing to a subsequence if necessary we again assume that $\{B_{p,k} u_k\}_{k \in \Ni}$ converges weakly to a $v \in L_p(\Ri^d)$.
Then $v = f - \lambda \, u$.
We will show that $B_p u = v$.
Indeed let $\phi \in C_c^\infty(\Ri^d)$.
Then $\lim_{n \to \infty} B_{p,n}^* \phi = B_p^* \phi$ strongly in $L_q(\Ri^d)$ and
\[
(v, \phi)
= \lim_{n \to \infty} (B_{p,n} u_n, \phi)
= \lim_{n \to \infty} (u_n, B_{p,n}^* \phi)
= (u, B_p^* \phi).
\]
Therefore $u \in D(B_p)$ and $B_p u = v$.
Hence $(\lambda + B_p) u = f$.
\\
{\bf Step 2:} We will show that $\overline{B_p|_{D(B_p) \cap W^{1,p}(\Ri^d)}} = B_p$, which implies $B_p$ is $m$-accretive.
\\
Clearly $\overline{D(B_p) \cap W^{1,p}(\Ri^d)}^{\|\cdot\|_{D(B_p)}} \subset D(B_p)$.
For the reverse inclusion let $u \in D(B_p)$ and $\lambda$ be defined as in Step 1.
Since $(\lambda + B_p) u \in L_p(\Ri^d)$ and $\overline{B_p|_{D(B_p) \cap W^{1,p}(\Ri^d)}}$ is $m$-accretive, there exists a $v \in \overline{D(B_p) \cap W^{1,p}(\Ri^d)}^{\|\cdot\|_{D(B_p)}}$ such that $(\lambda + B_p) v = (\lambda + B_p) u$.
Equivalently 
\begin{equation} \label{W Hq duality arguments}
(u-v, (\lambda + H_q) \phi) = 0
\end{equation}
for all $\phi \in C_c^\infty(\Ri^d)$.

Define $G_q = (B_p|_{C_c^\infty(\Ri^d)})^*$.
Then $H_q \subset G_q$.
Note that $|1 - \frac{2}{p}| < \cos\theta$ is equivalent to $|1 - \frac{2}{q}| < \cos\theta$.
Furthermore $C^*$ satisfies the same condition as those of $C$.
Therefore all previous results apply to $G_q$.
In particular, Proposition \ref{W Cc dense 1} gives $C_c^\infty(\Ri^d)$ is dense in $(D(G_q) \cap W^{1,q}(\Ri^d), \|\cdot\|_{D(G_q)})$ and Step 1 gives $\overline{G_q|_{D(G_q) \cap W^{1,q}(\Ri^d)}}$ is $m$-accretive.

Now it follows from \eqref{W Hq duality arguments} that
\[
(u-v, (\lambda + G_q) \phi) = 0
\]
for all $\phi \in \overline{(D(G_q) \cap W^{1,q}(\Ri^d), \|\cdot\|_{D(G_q)})}$.
Since $\overline{G_q|_{D(G_q) \cap W^{1,q}(\Ri^d)}}$ is $m$-accretive, we must have $u = v$.
\\
{\bf Step 3:} We will show that $C_c^\infty(\Ri^d)$ is a core for $B_p$.
\\
This follows immediately from Proposition \ref{W Cc dense 1} and Step 2.
\end{proof}

\section{The core property for $A_p$}

Let $p \in (1, \infty)$ be such that $|1 - \frac{2}{p}| < \cos\theta$.
Suppose $B_a = 0$.
In Section \ref{S4.3}, we proved that the contraction $C_0$-semigroup $S$ generated by $A$ extends consistently to a contraction $C_0$-semigroup $S^{(p)}$ on $L_p(\Ri^d)$.
Let $-A_p$ be the generator of $S^{(p)}$.
In this section we will show that the operator $A_p$ and $B_p$ are in fact the same.
Consequently the space of test functions $C_c^\infty(\Ri^d)$ is a core for $A_p$.
This is the content of Theorem \ref{main theorem higher dimensions}, which is also the main theorem of the paper.

\begin{prop} \label{W Ap = Bp}
Let $p \in (1, \infty)$ be such that $|1 - \frac{2}{p}| < \cos\theta$.
Suppose $B_a = 0$.
Then $A_p = B_p$.
\end{prop}

\begin{proof}
Let $u \in D(A) \cap D(A_p)$.
Then
\[
(A_p u, \phi) = (Au, \phi) = \gota(u, \phi) = (u, H_q \phi)
\]
for all $\phi \in C_c^\infty(\Ri^d)$.
It follows that $u \in D(B_p)$ and $B_p u = A_p u$.
In particular this implies that $D(A) \cap D(A_p) \subset D(B_p)$.
However $D(A) \cap D(A_p)$ is a core for $A_p$ and $B_p$ is closed.
Hence $D(A_p) \subset D(B_p)$.
On the other hand note that $-A_p$ generates a contraction $C_0$-semigroup.
Therefore $A_p$ is $m$-accretive.
By Proposition \ref{W Cc dense 2} the operator $B_p$ is also $m$-accretive.
Hence $A_p = B_p$ as required.
\end{proof}

Theorem \ref{main theorem higher dimensions} now follows immediately from the above proposition.

\begin{proof}[Proof of Theorem \ref{main theorem higher dimensions}]
By Proposition \ref{W Cc dense 2} the space $C_c^\infty(\Ri^d)$ is a core for $B_p$. 
Since $A_p = B_p$ by Proposition \ref{W Ap = Bp}, it follows that $C_c^\infty(\Ri^d)$ is also a core for $A_p$.
\end{proof}

\section{More sufficient conditions in $L_2$}

This section is motivated by the fact that $B_2$ is accretive on $W^{2,2}(\Ri^d)$ without the requirement that $B_a = 0$ (cf.\ Proposition \ref{W 1st ineq}).
In fact more is true.

\begin{prop}
We have
\[
\R (B_2 u, u) \geq 0
\]
for all $u \in W^{1,2}(\Ri^d) \cap D(B_2)$.
\end{prop}

\begin{proof}
Let $u \in W^{1,2}(\Ri^d) \cap D(B_2)$.
Then 
\begin{eqnarray*}
\R (B_2 u, u) 
& = & - \R \sum_{k,l=1}^d \int_{\Ri^d} (\D_l (c_{kl} \, \D_k u)) \, \overline{u}
	= \R \sum_{k,l=1}^d \int_{\Ri^d} c_{kl} \, (\D_k u) \, \D_l \overline{u}
\\
& = & \R \int_{\Ri^d} (C \, \nabla u, \nabla u)
	= \int_{\Ri^d} ((\R C) \, \nabla u, \nabla u) 
	\geq 0
\end{eqnarray*}
as claimed.
\end{proof}

Define the operator $Z = \overline{B_2|_{C_c^\infty(\Ri^d)}}$.
Then $Z$ is closed.
Furthermore we have the following.

\begin{prop} \label{W Z description}
The operator $Z$ is accretive and $Z = \overline{B_2|_{W^{1,2}(\Ri^d) \cap D(B_2)}}$.
\end{prop}

\begin{proof}
It suffices to show $Z = \overline{B_2|_{W^{1,2}(\Ri^d) \cap D(B_2)}}$.
This follows immediately from Proposition \ref{W Cc dense 1}.
\end{proof}

From now on we drop the condition that $B_a \neq 0$.
In this section we will provide many sufficient conditions for the space of test functions $C_c^\infty(\Ri^d)$ to be a core for the operator $A$.
Define the operator $L$ in $L_2(\Ri^d)$ as follows.
\begin{equation} \label{W L}
Lu = - \sum_{k,l=1}^d \D_k \big( \overline{(B_a)_{kl}} \, \D_l u \big)
\end{equation}
on the domain
\[
D(L) = C_c^\infty(\Ri^d).
\]
Next define the operator associated with $B_a$ as $(B_a)^\op = L^*$, which is the dual of $L$.
In what follows we denote $(\D_k B_a)_{kl} = \D_k \big( (B_a)_{kl} \big)$ for all $k, l \in \{1, \ldots, d\}$.
Although $(B_a)^\op$ appears to be a differential operator of second order, it is in fact a first-order differential operator.
Indeed for all $u \in D((B_a)^\op)$ and $\phi \in C_c^\infty(\Ri^d)$ we have
\begin{eqnarray}
\big( (B_a)^\op u, \phi \big)
& = & (u, L \phi)
	= - \sum_{k,l=1}^d \int_{\Ri^d} 
		u \, \D_k \big( (B_a)_{kl} \, \D_l \phi \big)
\nonumber
\\
& = & - \sum_{k,l=1}^d \int_{\Ri^d} 
		u \, \Big( (\D_k B_a)_{kl} \, \D_l \phi 
		+ (B_a)_{kl} \, \D_k \D_l \phi \Big)
\nonumber
\\
& = & - \sum_{k,l=1}^d \int_{\Ri^d} 
		u \, (\D_k B_a)_{kl} \, \D_l \phi,
\label{Ba op 1}
\end{eqnarray}
where the last step follows from the anti-symmetry of $B_a$.
Note that $(B_a)_{kl} \in W^{2,\infty}(\Ri^d)$ for all $k,l \in \{1, \ldots, d\}$.
Therefore it follows from \eqref{Ba op 1} that $W^{1,2}(\Ri^d) \subset D((B_a)^\op)$ and 
\begin{eqnarray*}
\big( (B_a)^\op u, \phi \big) 
& = & \sum_{k,l=1}^d \int_{\Ri^d} \D_l \big( (\D_k B_a)_{kl} \, u \big) \, \phi
	= \sum_{k,l=1}^d \int_{\Ri^d} \big( (\D_l \D_k B_a)_{kl} \, u + (\D_k B_a)_{kl} \, (\D_l u) \big) \, \phi
\\
& = & \sum_{k,l=1}^d \int_{\Ri^d} (\D_k B_a)_{kl} \, (\D_l u) \, \phi
	= - \sum_{k,l=1}^d \int_{\Ri^d} (\D_l B_a)_{kl} \, (\D_k u) \, \phi
\end{eqnarray*}
for all $u \in W^{1,2}(\Ri^d)$ and $\phi \in C_c^\infty(\Ri^d)$ since $B_a$ is anti-symmetric.
Hence 
\[
(B_a)^\op u 
= \sum_{k,l=1}^d \D_l \big( (\D_k B_a)_{kl} \, u \big)
= - \sum_{k,l=1}^d (\D_l B_a)_{kl} \, \D_k u
\]
for all $u \in W^{1,2}(\Ri^d)$.

\begin{lemm} \label{W Ba estimate}
For all $\varepsilon > 0$ there exists an $M > 0$ such that
\[
\big| \big( (B_a)^\op u, - \Delta u \big) \big|
\leq \varepsilon \int_{\Ri^d} \|(\D_l B_a) \, U\|_{HS}^2 + M \, \|\nabla u\|_2^2
\]
for all $u \in C_c^\infty(\Ri^d)$, where $U = (\D_l \D_k u)_{1 \leq k,l \leq d}$.
\end{lemm}

\begin{proof}
Let $u \in C_c^\infty(\Ri^d)$ and write $U = (\D_l \D_k u)_{1 \leq k,l \leq d}$.
Then
\begin{eqnarray*}
\big| \big( (B_a)^\op u, - \Delta u \big) \big|
& = & \Big| \sum_{k,l,j=1}^d \int_{\Ri^d} (\D_l B_a)_{kl} \, (\D_k u) \, \D_j^2 \overline{u} \Big|
\\
& = & \Big| \sum_{k,l,j=1}^d \int_{\Ri^d} \Big( (\D_j \D_l B_a)_{kl} \, \D_k u
	+ (\D_l B_a)_{kl} \, \D_k \D_j u \Big) \, \D_j \overline{u} \Big|
\\
& \leq & \Big| \sum_{k,l,j=1}^d \int_{\Ri^d} (\D_j \D_l B_a)_{kl} \, (\D_k u) \, \D_j \overline{u} \Big|
	+ \Big| \sum_{k,l,j=1}^d \int_{\Ri^d} (\D_l B_a)_{kl} \, (\D_k \D_j u) \, \D_j \overline{u} \Big|
\\*
& = & \hspace{30mm} ({\rm I}) \hspace{20mm} + \hspace{25mm} ({\rm II}).
\end{eqnarray*}
For (I) we have
\[
\Big| \sum_{k,l,j=1}^d \int_{\Ri^d} (\D_j \D_l B_a)_{kl} \, (\D_k u) \, \D_j \overline{u} \Big|
\leq d^2 \, \sup_{1 \leq k,l \leq d} \|(B_a)_{kl}\|_{W^{2,\infty}} \, \|\nabla u\|_2^2.
\]
We estimate the term (II) by
\begin{eqnarray*}
\Big| \sum_{k,l,j=1}^d \int_{\Ri^d} (\D_l B_a)_{kl} \, (\D_k \D_j u) \, \D_j \overline{u} \Big|
& = & \Big| \sum_{l,j=1}^d \int_{\Ri^d} \big( (\D_l B_a) \, U \big)_{lj} \, \D_j \overline{u} \Big|
\\
& \leq & \varepsilon \, \sum_{l,j=1}^d \int_{\Ri^d} \big| \big( (\D_l B_a) \, U \big)_{lj} \big|^2
	+ \frac{d}{4 \varepsilon} \, \|\nabla u\|_2^2
\\
& = & \varepsilon \int_{\Ri^d} \|(\D_l B_a) \, U\|_{HS}^2 + \frac{d}{4 \varepsilon} \, \|\nabla u\|_2^2.
\end{eqnarray*}
Hence 
\[
\big| \big( (B_a)^\op u, - \Delta u \big) \big|
\leq \varepsilon \int_{\Ri^d} \|(\D_l B_a) \, U\|_{HS}^2 + M \, \|\nabla u\|_2^2,
\]
where 
\[
M = d^2 \, \sup_{1 \leq k,l \leq d} \|(B_a)_{kl}\|_{W^{2,\infty}} + \frac{d}{4 \varepsilon}
\]
as required.
\end{proof}

\begin{lemm} \label{W Ba op lemma}
For all $\varepsilon > 0$ there exists an $M > 0$ such that
\[
\Big| \int_{\Ri^d} \tr (U \, B_a \, \overline{U})
\leq \varepsilon \int_{\Ri^d} \|(\D_l B_a) \, U\|_{HS}^2
	+ M \, \|\nabla u\|_2^2
\]
for all $u \in C_c^\infty(\Ri^d)$, where $U = (\D_l \D_k u)_{1 \leq k,l \leq d}$.
\end{lemm}

\begin{proof}
Let $u \in C_c^\infty(\Ri^d)$ and write $U = (\D_l \D_k u)_{1 \leq k,l \leq d}$.
Then
\begin{eqnarray*}
((B_a)^\op u, - \Delta u)
& = & \sum_{k,l,j=1}^d \int_{\Ri^d} \big( \D_l ((B_a)_{kl} \, \D_k u) \big) \, \D_j^2 \overline{u}
\\
& = & - \sum_{k,l,j=1}^d \int_{\Ri^d} \Big( \D_l \big( (\D_j B_a)_{kl} \, \D_k u 
	+ (B_a)_{kl} \, \D_j \D_k u \big) \Big) \, \D_j \overline{u}
\\
& = & - \sum_{k,l,j=1}^d \int_{\Ri^d} (\D_l \D_j B_a)_{kl} \, (\D_k u) \, \D_j \overline{u} 
	+ (\D_j B_a)_{kl} \, (\D_l \D_k u) \, \D_j \overline{u}
	\\
	& & + \sum_{k,l,j=1}^d \int_{\Ri^d} (B_a)_{kl} \, (\D_j \D_k u) \, \D_l \D_j \overline{u}
\\
& = & - \sum_{k,l,j=1}^d \int_{\Ri^d} (\D_l \D_j B_a)_{kl} \, (\D_k u) \, \D_j \overline{u} 
	+ \int_{\Ri^d} \tr (U \, B_a \, \overline{U}),
\end{eqnarray*}
where in the last step we used $\sum_{k,l=1}^d (\D_j B_a)_{kl} \, (\D_l \D_k u) = 0$ for all $j \in \{1, \ldots, d\}$, which follows from the anti-symmetry of $B_a$.

Let $\varepsilon > 0$ and $M$ be as in Lemma \ref{W Ba estimate}.
Then
\begin{eqnarray*}
\Big| \int_{\Ri^d} \tr (U \, B_a \, \overline{U}) \Big|
& \leq & |((B_a)^\op u, - \Delta u)| + \Big| \sum_{k,l,j=1}^d \int_{\Ri^d} (\D_l \D_j B_a)_{kl} \, (\D_k u) \, \D_j \overline{u}  \Big|
\\
& \leq & \varepsilon \int_{\Ri^d} \|(\D_l B_a) \, U\|_{HS}^2
	+ (M + d^2 \, \|B_a\|_{W^{2,\infty}}) \, \|\nabla u\|_2^2,
\end{eqnarray*}
where we used Lemma \ref{W Ba estimate} in the last step.
\end{proof}

\begin{lemm} \label{W commutator}
Let $u \in C_c^\infty(\Ri^d)$.
Then
\begin{eqnarray*}
\R \big( (B_2 - B_2^*)u, - \Delta u \big)
& = & 2 \, \I \sum_{k,l,j=1}^d \int_{\Ri^d} (\D_l \D_j \I C)_{kl} \, (\D_k u) \, \D_j \overline{u}
\\
& & {} + 2 \, \I \sum_{j=1}^d \int_{\Ri^d} \tr ((\D_j \I C) \, U) \, \, \D_j \overline{u}.
\end{eqnarray*}
\end{lemm}

\begin{proof}
Let $u \in C_c^\infty(\Ri^d)$ and write $U = (\D_l \D_k u)_{1 \leq k,l \leq d}$.
Then
\begin{eqnarray*}
\big( (B_2 - B_2^*)u, - \Delta u \big)
& = & \sum_{k,l,j=1}^d \int_{\Ri^d} \Big( \D_l \big( (c_{kl} - \overline{c_{lk}}) \, \D_k u \big) \Big) \, 
	\D_j^2 \overline{u}
\\
& = & 2i \sum_{k,l,j=1}^d \int_{\Ri^d} \Big( \D_l \big( (\I C)_{kl} \, \D_k u \big) \Big) \, 
	\D_j^2 \overline{u}
\\
& = & - 2i \sum_{k,l,j=1}^d \int_{\Ri^d} \Big( \D_l \big( (\D_j \I C)_{kl} \, \D_k u 
	+ (\I C)_{kl} \, \D_j \D_k u \big) \Big) \, \D_j \overline{u}
\\
& = & - 2i \sum_{k,l,j=1}^d \int_{\Ri^d} \Big( (\D_l \D_j \I C)_{kl} \, (\D_k u) 
	+ (\D_j \I C)_{kl} \, (\D_l \D_k u) \Big) \, \D_j \overline{u}
	\\
	& &
	{} + 2i \sum_{k,l,j=1}^d \int_{\Ri^d} (\I C)_{kl} \, (\D_j \D_k u) \, (\D_l \, \D_j \overline{u})
\\
& = & - 2i \sum_{k,l,j=1}^d \int_{\Ri^d} (\D_l \D_j \I C)_{kl} \, (\D_k u) \, \D_j \overline{u}
	- 2i \, \sum_{j=1}^d \int_{\Ri^d} \tr ((\D_j \I C) \, U) \, \, \D_j \overline{u}
	\\
	& &
	{} + 2i \int_{\Ri^d} \tr (U \, (\I C) \, \overline{U}).
\end{eqnarray*}
Taking the real parts both sides gives the statement since $\tr (U \, (\I C) \, \overline{U}) \in \Ri$.
\end{proof}

\begin{lemm} \label{W tr(DjCU) estimate}
Let $u \in C_c^\infty(\Ri^d)$.
Then
\begin{eqnarray*}
\R \sum_{j=1}^d \int_{\Ri^d} \tr \big( (\D_j C) \, U \big) \, \D_j \overline{u}
& = & \frac{1}{2} \, \R \sum_{k,l,j=1}^d \int_{\Ri^d} (\D_j^2 c_{kl}) \, (\D_l u) \, \D_k \overline{u}
	- 2 \, (\D_k \D_j c_{kl}) \,  (\D_l u) \, \D_j \overline{u}
	\\
	& & {} + \I \sum_{k,l,j=1}^d \int_{\Ri^d} (\D_l \D_j \I C)_{kl} \, (\D_j u) \, \D_k \overline{u} 
	\\
	& & {} + \I \sum_{j=1}^d \int_{\Ri^d} \tr ((\D_j \I C) \, \overline{U}) \, \, \D_j u.
\end{eqnarray*}
\end{lemm}

\begin{proof}
Let $u \in C_c^\infty(\Ri^d)$ and write $U = (\D_l \D_k u)_{1 \leq k,l \leq d}$.
Then
\begin{eqnarray*}
(B_2 u, - \Delta u)
& = & \sum_{k,l,j=1}^d \int_{\Ri^d} (\D_l (c_{kl} \, \D_k u) \, \D_j^2 u
	= - \sum_{k,l,j=1}^d \int_{\Ri^d} \Big( \D_l ((\D_j c_{kl}) \, \D_k u 
		+ c_{kl} \, \D_j \D_k u) \Big) \, \D_j \overline{u}
\\
& = & \sum_{k,l,j=1}^d \int_{\Ri^d} ((\D_j c_{kl}) \, \D_k u 
	+ c_{kl} \, \D_j \D_k u) \, \D_l \D_j \overline{u}
\\
& = & - \sum_{k,l,j=1}^d \int_{\Ri^d} \Big( (\D_j^2 c_{kl}) \, \D_k u + (\D_j c_{kl}) \, \D_j \D_k u
	+ (\D_j c_{kl}) \, \D_j \D_k u + c_{kl} \, \D_j^2 \D_k u \Big) \, \D_l \overline{u}
\\
& = & - \sum_{k,l,j=1}^d \int_{\Ri^d} (\D_j^2 c_{kl}) \, (\D_k u) \, \D_l \overline{u}
	+ 2 \, (\D_j c_{kl}) \, (\D_j \D_k u) \, \D_l \overline{u} 
	- (\D_j^2 u) \, \D_k (c_{kl} \, \D_l \overline{u})
\\
& = & -\sum_{k,l,j=1}^d \int_{\Ri^d} (\D_j^2 c_{kl}) \, (\D_k u) \, \D_l \overline{u}
	- 2 \, (\D_j u) \Big( (\D_k \D_j c_{kl}) \,  \D_l \overline{u} 
	+ (\D_j c_{kl}) \, (\D_k \D_l \overline{u}) \Big)
	\\
	& & {} + (-\Delta u, B_2^* u).
\end{eqnarray*}
Hence 
\begin{eqnarray*}
\sum_{k,l,j=1}^d \int_{\Ri^d} (\D_j c_{kl}) \, (\D_l \D_k \overline{u}) \, \D_j u
& = & \frac{1}{2} \sum_{k,l,j=1}^d \int_{\Ri^d} (\D_j^2 c_{kl}) \, (\D_k u) \, \D_l \overline{u}
	- 2 \, (\D_k \D_j c_{kl}) \,  (\D_l \overline{u}) \, (\D_j u)
	\\
	& & {} + \frac{1}{2} \, \Big( (B_2 u, - \Delta u) - (-\Delta u, B_2^* u) \Big).
\end{eqnarray*}
Replacing $u$ by $\overline{u}$ in the above equation and taking the real parts on both sides gives
\begin{eqnarray*}
\R \sum_{j=1}^d \int_{\Ri^d} \tr \big( (\D_j C) \, U \big) \, \D_j \overline{u}
& = & \R \sum_{k,l,j=1}^d \int_{\Ri^d} (\D_j c_{kl}) \, (\D_l \D_k u) \, \D_j \overline{u}
\\
& = & \frac{1}{2} \, \R \sum_{k,l,j=1}^d \int_{\Ri^d} (\D_j^2 c_{kl}) \, (\D_k \overline{u}) \, \D_l u
	- 2 \, (\D_k \D_j c_{kl}) \,  (\D_l u) \, (\D_j \overline{u})
	\\*
	& & {} + \frac{1}{2} \, \R \big( (B_2 - B_2^*) \overline{u}, - \Delta \overline{u} \big).
\end{eqnarray*}
Using Lemma \ref{W commutator} we yield the result.
\end{proof}

\begin{prop} \label{W Z m-accretive}
Suppose one of the following holds.
\begin{tabeleq}
\item The matrix $B_s$ has constant entries.
\item There exist $\theta_1, \theta_2 \in [0, \frac{\pi}{2})$, $\phi \in W^{2, \infty}(\Ri^d)$ and a $d \times d$ matrix $\widetilde{C}$ with entries in $W^{2, \infty}(\Ri^d)$ such that $\theta = \theta_1 + \theta_2$, $\phi(x) \in \Sigma_{\theta_1}$ for all $x \in \Ri^d$, $\widetilde{C}$ takes values in $\Sigma_{\theta_2}$ and $C = \phi \, \widetilde{C}$. 
Write $\widetilde{C} = \widetilde{R} + i \, \widetilde{B}$, where $\widetilde{R}$ and $\widetilde{B}$ are $d \times d$ matrix-valued functions with real-valued entries.
Set $\widetilde{R}_s = \frac{1}{2} \, (\widetilde{R} + \widetilde{R}^T)$.
Also define $\R \widetilde{C} = \frac{1}{2} \, \big( \widetilde{C} + (\widetilde{C})^* \big)$.
Suppose further that there exists an $h > 0$ such that 
\[
\tr (U \, (\R \widetilde{C}) \, \overline{U}) \geq h \, \tr (U \, \widetilde{R}_s \, \overline{U})
\]
for all $u \in C_c^\infty(\Ri^d)$, where $U = (\D_l \D_k u)_{1 \leq k,l \leq d}$.
\item There exists an $M > 0$ such that $\|(\D_l B_a) \, U\|_{HS}^2 \leq M \, \tr (U \, R_s \, \overline{U})$ for all $l \in \{1, \ldots, d\}$ and $u \in C_c^\infty(\Ri^d)$, where $U = (\D_l \D_k u)_{1 \leq k,l \leq d}$.
\end{tabeleq}
Then $Z$ is $m$-accretive.
\end{prop}

\begin{proof}
By Proposition \ref{W Z description} we have that that $D(-\Delta) = W^{2,2}(\Ri^d) \subset D(Z)$.
We will show that there exists a $\beta \in \Ri$ such that 
\begin{equation} \label{W T1.50}
\R (Zu, - \Delta u) \geq - \beta \, \|\nabla u\|_2^2
\end{equation}
for all $u \in D(-\Delta) = W^{2,2}(\Ri^d)$.
It then follows from \cite[Theorem 1.50]{Ouh5} that $Z$ is $m$-accretive.
Since $C_c^\infty(\Ri^d)$ is dense in $W^{2,2}(\Ri^d)$ and is a core for $Z$, it suffices to show that \eqref{W T1.50} holds for all $u \in C_c^\infty(\Ri^d)$.

Let $u \in C_c^\infty(\Ri^d)$ and $U = (\D_l \D_k u)_{1 \leq k,l \leq d}$.
Using integration by parts we obtain
\begin{eqnarray*}
(Zu, - \Delta u)
& = & \sum_{k,l,j=1}^d \int_{\Ri^d} \big( \D_l (c_{kl} \, \D_k u) \big) \, \D_j^2 \overline{u}
	= - \sum_{k,l,j=1}^d \int_{\Ri^d} \Big( \D_l \big( (\D_j c_{kl}) \, (\D_k u) 
	+ c_{kl} \, (\D_j \D_k u) \big) \Big) \, \D_j \overline{u}
\\
& = & - \sum_{k,l,j=1}^d \int_{\Ri^d} (\D_l \D_j c_{kl}) \, (\D_k u) \, \D_j \overline{u}
	+ (\D_j c_{kl}) \, (\D_l \D_k u) \, \D_j \overline{u} - c_{kl} \, (\D_j \D_k u) \, \D_l \D_j \overline{u}
\\
& = & - \sum_{k,l,j=1}^d \int_{\Ri^d} (\D_l \D_j c_{kl}) \, (\D_k u) \, \D_j \overline{u}
	- \sum_{j=1}^d \int_{\Ri^d} \tr \big( (\D_j C) \, U \big) \, \D_j \overline{u}
	+ \int_{\Ri^d} \tr (U \, C \, \overline{U}).
\end{eqnarray*}
Therefore
\begin{eqnarray*}
\R (Zu, - \Delta u)
& = & - \R \sum_{k,l,j=1}^d \int_{\Ri^d} (\D_l \D_j c_{kl}) \, (\D_k u) \, \D_j \overline{u}
	- \R \sum_{j=1}^d \int_{\Ri^d} \tr \big( (\D_j C) \, U \big) \, \D_j \overline{u}
	\\
	& & {} + \int_{\Ri^d} \tr \big( U \, (\R C) \, \overline{U} \big)
\\
& = & ({\rm I}) + ({\rm II}) + ({\rm III}).
\end{eqnarray*}
The estimate for (I) is straightforward as
\begin{equation} \label{S4.6 term (I)}
({\rm I}) 
\geq - \sum_{k,l,j=1}^d \int_{\Ri^d} \big| (\D_l \D_j c_{kl}) \, (\D_k u) \, \D_j \overline{u} \big|
\geq - M_1 \|\nabla u\|_2^2,
\end{equation}
where $M_1 = d^2 \, \sup_{1 \leq k,l \leq d} \|c_{kl}\|_{W^{2,\infty}}$.
The estimates for (II) and (III) are more involved.
We consider three cases according to the three conditions (i), (ii) and (iii) imposed above.
\\
{\bf Case 1:} Suppose (i) holds.
\\
Since $U = U^T$ and $R_a = - R_a^T$, we have
\[
\tr \big( (\D_j R_a) \, U \big)
= \tr \big( U^T \, (\D_j R_a)^T \big)
= - \tr \big( U \, (\D_j R_a) \big)
= - \tr \big( (\D_j R_a) \, U \big).
\]
Therefore $\tr \big( (\D_j R_a) \, U \big) = 0$.
This implies
\[
\tr \big( (\D_j \I C) \, U \big) 
= \tr \big( (\D_j B_s) \, U \big) - i \, \tr \big( (\D_j R_a) \, U \big)
= \tr \big( (\D_j B_s) \, U \big) 
= 0,
\]
where the last equality follows from the hypothesis.
Using Lemma \ref{W tr(DjCU) estimate} we obtain that
\[
({\rm II}) = \sum_{k,l,j=1}^d \int_{\Ri^d} \R \Big( \frac{1}{2} \, (\D_j^2 c_{kl}) \, (\D_l u) \, \D_k \overline{u}
	- (\D_k \D_j c_{kl}) \,  (\D_l u) \, \D_j \overline{u} \Big)
	+ \I \Big( (\D_l \D_j \I C)_{kl} \, (\D_j u) \, \D_k \overline{u} \Big).
\]
Consequently
\[
({\rm II}) \geq - M_2 \, \|\nabla u\|_2^2,
\]
where $M_2 = 3 \, d^2 \, \sup_{1 \leq k,l \leq d} \|c_{kl}\|_{W^{2,\infty}}$.
Note that $({\rm III}) \geq 0$.
Hence 
\[
\R (Zu, - \Delta u) \geq - (M_1 + M_2) \, \|\nabla u\|_2^2.
\]
%
{\bf Case 2:} Suppose (ii) holds.
\\
We first consider (II).
We have
\begin{eqnarray*}
({\rm II})
& = & - \R \sum_{j=1}^d \int_{\Ri^d} \tr \big( \D_j (\phi \, \widetilde{C}) \, U \big) \, \D_j \overline{u}
\\
& = & - \R \sum_{j=1}^d \int_{\Ri^d} (\D_j \phi) \, \tr (\widetilde{C} \, U) \, \D_j \overline{u}
	- \R \sum_{j=1}^d \int_{\Ri^d} \phi \, \tr \big( (\D_j \widetilde{C}) \, U \big) \, \D_j \overline{u}
\\
& = & \hspace{25mm} ({\rm IIa}) \hspace{20mm} + \hspace{25mm} ({\rm IIb}).
\end{eqnarray*}
Let 
\[
M_3 = 64 \, d \, (1 + \tan\theta_1)^2 \, (1 + \tan\theta_2)^2 \, 
		\|\widetilde{R}_s\|_\infty \, \sup_{1 \leq j \leq d} \|\D_j^2 \phi\|_\infty
\]
and 
\[
M_4 = 32 \, d^2 \, (1 + \tan\theta_1) \, (1 + \tan\theta_2)^2 \, 
		\sup_{1 \leq j \leq d} \|\D_j^2 \widetilde{C}\|_\infty.
\]
Let
\[
\varepsilon = \frac{(1 - \tan\theta_1 \, \tan\theta_2) \, h}{4 \, (M_3 \vee M_4 \vee 1)}.
\]
Note that $\varepsilon > 0$ as $1 - \tan\theta_1 \, \tan\theta_2 > 0$.
Indeed, if $\tan\theta = 0$ then $\theta = \theta_1 = \theta_2 = 0$, which implies $1 - \tan\theta_1 \, \tan\theta_2 = 1 > 0$.
If $\tan\theta > 0$ then $1 - \tan\theta_1 \, \tan\theta_2 = \frac{\tan\theta_1 + \tan\theta_2}{\tan\theta} > 0$.

For (IIa) we estimate
\begin{eqnarray*}
({\rm IIa})
& \geq & - \varepsilon \int_{\Ri^d} \Big( \sum_{j=1}^d |\D_j \phi|^2 \Big) \, |\tr (\widetilde{C} \, U)|^2
	- \frac{1}{4 \varepsilon} \, \|\nabla u\|_2^2.
\end{eqnarray*}
Note that 
\[
|\D_j \phi|^2 \leq 4 \, (1 + \tan\theta_1)^2 \, \sup_{1 \leq j \leq d} \|\D_j^2 \phi\|_\infty \, \R \phi
\]
for all $j \in \{1\, \ldots, d\}$ by Lemma \ref{W f sectorial}.
Moreover,
\[
|\tr (\widetilde{C} \, U)|^2
\leq d \, \|\widetilde{C} \, U\|_{HS}^2
\leq 16 \, d \, (1 + \tan\theta_2)^2 \, \|\widetilde{R}_s\|_\infty \, \tr (U \, \widetilde{R}_s \, \overline{U}),
\]
where we used Lemma \ref{W HS norm lemma} in the last step.
Consequently
\begin{eqnarray}
({\rm IIa})
& \geq & -\varepsilon \, M_3 \int_{\Ri^d} (\R \phi) \, \tr (U \, \widetilde{R}_s \, \overline{U})
	- \frac{1}{4 \varepsilon} \, \|\nabla u\|_2^2
\nonumber
\\
& \geq & - \frac{(1 - \tan\theta_1 \, \tan\theta_2) \, h}{4} \, \int_{\Ri^d} (\R \phi) \, \tr (U \, \widetilde{R}_s \, \overline{U})
	- \frac{1}{4 \varepsilon} \, \|\nabla u\|_2^2.
\label{S4.6 term (IIa)}
\end{eqnarray}

For (IIb) we estimate as follows.
Since $\phi(x) \in \Sigma_{\theta_1}$ for all $x \in \Ri^d$, we have 
\[
|\phi| \leq |\R \phi| + |\I \phi| \leq (1 + \tan\theta_1) \, \R \phi.
\]
Therefore
\begin{eqnarray}
({\rm IIb})
& \geq & - \sum_{j=1}^d \int_{\Ri^d} |\phi| \, \big| \tr \big( (\D_j \widetilde{C}) \, U \big) \big| \, 
	|\D_j \overline{u}|
\nonumber
\\
& \geq & - (1 + \tan\theta_1) \, \sum_{j=1}^d \int_{\Ri^d} (\R \phi) \, 
	\Big( \varepsilon \, \big| \tr \big( (\D_j \widetilde{C}) \, U \big) \big|^2 
	+ \frac{1}{4 \varepsilon} \, |\D_j \overline{u}|^2 \Big)
\nonumber
\\
& \geq & - \varepsilon \, (1 + \tan\theta_1) \, \sum_{j=1}^d \int_{\Ri^d} (\R \phi) \, 
	\big| \tr \big( (\D_j \widetilde{C}) \, U \big) \big|^2 
	- \frac{(1+\tan\theta_1) \, \|\phi\|_\infty}{4 \varepsilon} \, \|\nabla u\|_2^2
\nonumber
\\
& \geq & - \varepsilon \, M_4 \, \int_{\Ri^d} (\R \phi) \, \tr (U \, \widetilde{R}_s \, \overline{U})
	- \frac{(1+\tan\theta_1) \, \|\phi\|_\infty}{4 \varepsilon} \, \|\nabla u\|_2^2
\nonumber
\\
& \geq & - \frac{(1 - \tan\theta_1 \, \tan\theta_2) \, h}{4} \, \int_{\Ri^d} (\R \phi) \, \tr (U \, \widetilde{R}_s \, \overline{U})
	- \frac{(1+\tan\theta_1) \, \|\phi\|_\infty}{4 \varepsilon} \, \|\nabla u\|_2^2,
\hspace*{7mm}
\label{S4.6 term (IIb)}
\end{eqnarray}
where we used Corollary \ref{W Oleinik 2}(a) in the fourth step.

On the other hand, estimating (III) gives
\[
({\rm III}) 
= \int_{\Ri^d} \tr \big( U \, (\R (\phi \, \widetilde{C})) \, \overline{U} \big)
= \int_{\Ri^d} (\R \phi) \, \tr (U \, (\R \widetilde{C}) \, \overline{U})
	- (\I \phi) \, \tr (U \, (\I \widetilde{C}) \, \overline{U}).
\]
Since $\phi(x) \in \Sigma_{\theta_1}$ for all $x \in \Ri^d$, we have $|\I \phi| \leq (\tan\theta_1) \, \R \phi$.
Also as $\widetilde{C}$ takes values in $\Sigma_{\theta_2}$, we deduce that $|(\I (\widetilde{C} \, U e_j, U e_j)| \leq (\tan\theta_2) \, \R (\widetilde{C} \, U e_j, U e_j)$, which in turns implies that $|\tr (U \, (\I \widetilde{C}) \, \overline{U})| \leq (\tan\theta_2) \, \tr (U \, (\R \widetilde{C}) \, \overline{U})$.
Therefore
\begin{eqnarray}
({\rm III}) 
& \geq & \int_{\Ri^d} (1 - \tan\theta_1 \, \tan\theta_2) \, 
	(\R \phi) \, \tr (U \, (\R \widetilde{C}) \, \overline{U})
\nonumber
\\
& \geq & \int_{\Ri^d} (1 - \tan\theta_1 \, \tan\theta_2) \, h \, 
	(\R \phi) \, \tr (U \, \widetilde{R}_s \, \overline{U})
\label{S4.6 term (III)}
\end{eqnarray}
by the hypothesis.
Hence by \eqref{S4.6 term (I)}, \eqref{S4.6 term (IIa)}, \eqref{S4.6 term (IIb)} and \eqref{S4.6 term (III)} we have
\begin{eqnarray*}
\R (Zu, - \Delta u) 
& \geq & \frac{(1 - \tan\theta_1 \, \tan\theta_2) \, h}{2} \int_{\Ri^d} (\R \phi) \, \tr (U \, \widetilde{R}_s \, \overline{U})
	\\
	& & {} - \big( M_1 + \frac{1+(1+\tan\theta_1) \, \|\phi\|_\infty}{4 \varepsilon} \big) \, \|\nabla u\|_2^2
\\
& \geq & - \big( M_1 + \frac{1+(1+\tan\theta_1) \, \|\phi\|_\infty}{4 \varepsilon} \big) \, \|\nabla u\|_2^2.
\end{eqnarray*}
{\bf Case 3:} Suppose (iii) holds.
\\
Let $\varepsilon_1 = \frac{1}{2M}$ and $M'$ be corresponding to $\varepsilon_1$ as in Lemma \ref{W Ba op lemma}.
Then
\begin{eqnarray*}
({\rm III}) 
& = & \int_{\Ri^d} \tr (U \, R_s \, \overline{U}) + i \int_{\Ri^d} \tr (U \, B_a \, \overline{U})
\\
& \geq & \int_{\Ri^d} \tr (U \, R_s \, \overline{U}) 
	- \varepsilon_1 \int_{\Ri^d} \|(\D_l B_a) \, U\|_{HS}^2 - M' \, \|\nabla u\|_2^2
\\
& \geq & \frac{1}{2} \int_{\Ri^d} \tr (U \, R_s \, \overline{U}) - M' \, \|\nabla u\|_2^2
\end{eqnarray*}
since $\|(\D_l B_a) \, U\|_{HS}^2 \leq M \, \tr (U \, R_s \, \overline{U})$ by hypothesis.

Let $\varepsilon_2 = \frac{1}{4dM''}$, where $M''$ is the constant as in Corollary \ref{W Oleinik 2}(a).
Then
\[
({\rm II}) 
\geq - \varepsilon_2 \, \sum_{j=1}^d \int_{\Ri^d} \big| \tr \big( (\D_j C) \, U \big) \big|^2
	- \frac{1}{4 \varepsilon_2} \, \|\nabla u\|_2^2
\geq - \frac{1}{4} \, \int_{\Ri^d} \tr (U \, R_s \, \overline{U})
	- \frac{1}{4 \varepsilon_2} \, \|\nabla u\|_2^2,
\]
where we used Corollary \ref{W Oleinik 2}(a) in the last step.
Hence
\[
\R (Zu, - \Delta u) 
\geq \frac{1}{4} \int_{\Ri^d} \tr (U \, R_s \, \overline{U}) 
	- (\frac{1}{4\varepsilon_2} + M_1 + M') \, \|\nabla u\|_2^2.
\]
The proof is complete.
\end{proof}

We emphasise that it is not known yet whether $B_2$ is accretive if $B_a \neq 0$.
The following theorem is of main interest and will be used extensively.

\begin{thrm} \label{W main theorem L2}
Suppose one of the following holds.
\begin{tabeleq}
\item \label{Bs=0} The matrix $B_s$ has constant entries.
\item \label{C=phi C tilde} There exist $\theta_1, \theta_2 \in [0, \frac{\pi}{2})$, $\phi \in W^{2, \infty}(\Ri^d)$ and a $d \times d$ matrix $\widetilde{C}$ with entries in $W^{2, \infty}(\Ri^d)$ such that $\theta = \theta_1 + \theta_2$, $\phi(x) \in \Sigma_{\theta_1}$ for all $x \in \Ri^d$, $\widetilde{C}$ takes values in $\Sigma_{\theta_2}$ and $C = \phi \, \widetilde{C}$. 
Write $\widetilde{C} = \widetilde{R} + i \, \widetilde{B}$, where $\widetilde{R}$ and $\widetilde{B}$ are $d \times d$ matrix-valued functions with real-valued entries.
Set $\widetilde{R}_s = \frac{1}{2} \, (\widetilde{R} + \widetilde{R}^T)$.
Also define $\R \widetilde{C} = \frac{1}{2} \, \big( \widetilde{C} + (\widetilde{C})^* \big)$.
Suppose further that there exists an $h > 0$ such that 
\[
\tr (U \, (\R \widetilde{C}) \, \overline{U}) \geq h \, \tr (U \, \widetilde{R}_s \, \overline{U})
\]
for all $u \in C_c^\infty(\Ri^d)$, where $U = (\D_l \D_k u)_{1 \leq k,l \leq d}$.
\item \label{HS norm on Ba} There exists an $M > 0$ such that $\|(\D_l B_a) \, U\|_{HS}^2 \leq M \, \tr (U \, R_s \, \overline{U})$ for all $l \in \{1, \ldots, d\}$ and $u \in C_c^\infty(\Ri^d)$, where $U = (\D_l \D_k u)_{1 \leq k,l \leq d}$.
\end{tabeleq}
Then $A = B_2 = Z$.
Moreover, $C_c^\infty(\Ri^d)$ is a core for $A$.
\end{thrm}

\begin{proof} 
By Proposition \ref{W Z m-accretive} the operator $Z$ is $m$-accretive.
We will show that $Z = B_2$.
Clearly $Z \subset B_2$.
For the reverse inclusion let $u \in D(B_2)$.
Then $(I + B_2) u \in L_2(\Ri^d)$.
Since $Z$ is $m$-accretive, there exists a $v \in D(Z)$ such that $(I + Z) v = (I + B_2) u$.
But $B_2$ is an extension of $Z$.
Therefore $(I + B_2) v = (I + B_2) u$.
Let $\phi \in C_c^\infty(\Ri^d)$.
Then
\[
0 = ((I + B_2)(u-v), \phi) = (u-v, (I + H_2) \phi),
\]
where $H_2$ is defined by \eqref{W Hq}.
But $H_2$ satisfies the same criteria as those of $B_2|_{C_c^\infty(\Ri^d)}$.
Therefore analogous arguments give that $\overline{H_2}$ is also $m$-accretive.
Consequently $u = v$.
Hence $Z = B_2$.
It follows that $B_2$ is $m$-accretive.
In particular $B_2$ is accretive.
Note that $A$ is $m$-accretive and $A \subset B_2$.
Therefore we must have $A = B_2 = Z$.
Moreover, since $C_c^\infty(\Ri^d)$ is a core for $Z$, it is also a core for $A$.
\end{proof}

The next proposition provides three easy criteria to verify Condition \ref{HS norm on Ba} in Theorem \ref{W main theorem L2}.

\begin{prop}
Suppose $C$ satisfies one of the following.
\begin{tabel}
\item There exists an $r \in \Ri \setminus \{0\}$ such that $R_s + i r \, \D_l B_a \geq 0$ for all $l \in \{1, \ldots, d\}$.
\item The matrices $R_s$ and $\D_l B_a$ commute for all $l \in \{1, \ldots, d\}$.
\item There exist a real-valued function $\phi \in W^{2, \infty}(\Ri^d)$ which satisfies $\phi \geq 0$ and a $d \times d$ matrix $\widetilde{C}$ which has constant entries and takes values in $\Sigma_{\theta}$ such that $C = \phi \, \widetilde{C}$. 
\end{tabel}
Then there exists an $M  > 0$ such that $\|(\D_l B_a) \, U\|_{HS}^2 \leq M \, \tr (U \, R_s \, \overline{U})$ for all $l \in \{1, \ldots, d\}$ and $u \in C_c^\infty(\Ri^d)$, where $U = (\D_l \D_k u)_{1 \leq k,l \leq d}$.
\end{prop}

\begin{proof}
Let $l \in \{1, \ldots, d\}$.
Let $u \in C_c^\infty(\Ri^d)$ and $U = (\D_l \D_k u)_{1 \leq k,l \leq d}$.

We first deal with (a) and (b).
Set $P = \sqrt{U \, U^*} \geq 0$.
Let $V$ be a unitary matrix such that $P = V \, D_P \, V^*$, where $D_P$ is a positive diagonal matrix.
Then

\pagebreak[1]

\begin{eqnarray*}
\|(\D_l B_a) \, U\|_{HS}^2
& = & - \tr (U^* \, (\D_l B_a)^2 \, U)
	= - \tr ((\D_l B_a)^2 \, P^2)
	= - \tr ((\D_l B_a)^2 \, V \, D_P^2 \, V^*)
\\
& = & - \tr (V^* \, (\D_l B_a)^2 \, V \, D_P^2)
	= \sum_{k=1}^d |(V^* \, (\D_l B_a)^2 \, V)_{kk}| \, |(D_P)_{kk}|^2.
\end{eqnarray*}
We consider two cases.
\\
{\bf Case 1:} Suppose (a) holds.
\\
Then $|((\D_l B_a) \, \xi, \xi)| \leq \frac{1}{|r|} \, (R_s \, \xi, \xi)$ for all $\xi \in \Ci^d$.
By Lemma \ref{W Q xi square general} we have $\|(\D_l B_a) \, \xi\|^2 \leq \frac{4}{r^2} \, \|R_s\|_\infty \, (R_s \, \xi, \xi)$ for all $\xi \in \Ci^d$.
In particular $\|(\D_l B_a) \, V e_k\|^2 \leq \frac{4}{r^2} \, \|R_s\|_\infty \, (V^* \, R_s \, V)_{kk}$ for all $k \in \{1, \ldots, d\}$.
It follows that 
\begin{eqnarray*}
\|(\D_l B_a) \, U\|_{HS}^2
& \leq & \frac{4}{r^2} \, \|R_s\|_\infty \, \sum_{k=1}^d (V^* \, R_s \, V)_{kk} \, |(D_P)_{kk}|^2
	= \frac{4}{r^2} \, \|R_s\|_\infty \, \tr (V^* \, R_s \, V \, D_P^2)
\\
& = & \frac{4}{r^2} \, \|R_s\|_\infty \, \tr (R_s \, P^2)
	= \frac{4}{r^2} \, \|R_s\|_\infty \, \tr (U^* \, R_s \, U)
\\
& = & \frac{4}{r^2} \, \|R_s\|_\infty \, \tr (U \, R_s \, U^*)
= \frac{4}{r^2} \, \|R_s\|_\infty \, \tr (U \, R_s \, \overline{U}),
\end{eqnarray*}
where the last equality follows from the fact that $U = U^T$.
\\
{\bf Case 2:} Suppose (b) holds.
\\
Let $W$ be a unitary matrix such that $\D_l B_a = W \, D \, W^*$, where $D$ is diagonal.
Therefore
\[
|D_{kk}|^2 
= |(W^* \, (\D_l B_a) \, W)_{kk}|^2 
\leq 2 \, \sup_{1 \leq l \leq d} \|\D_l^2 C\|_\infty \, (W^* \, R_s \, W)_{kk}
\]
for all $k \in \{1, \ldots, d\}$ by Lemma \ref{W Dl Ba < Rs}.
Since $R_s$ and $\D_l B_a$ commute, we may assume without loss of generality that the matrix $W$ also diagonalises $R_s$.
It follows that
\begin{eqnarray*}
|(V^* \, (\D_l B_a)^2 \, V)_{kk}|
& = & |(V^* \, W \, D^2 \, W^* \, V)_{kk}|
	= |((W^* \, V)^* \, D^2 \, W^* \, V)_{kk}|
\\
& = & \sum_{j=1}^d \big( (W^* \, V)^* \big)_{kj} \, |D_{jj}|^2 \, (W^* \, V)_{jk}
\\
& \leq & 2 \, \sup_{1 \leq l \leq d} \|\D_l^2 C\|_\infty \, 
	\sum_{j=1}^d \big( (W^* \, V)^* \big)_{kj} \, (W^* \, R_s \, W)_{jj} \, (W^* \, V)_{jk}
\\
& = & 2 \, \sup_{1 \leq l \leq d} \|\D_l^2 C\|_\infty \, (V^* \, R_s \, V)_{kk}
\end{eqnarray*}
for all $k \in \{1, \ldots, d\}$.
Hence
\begin{eqnarray*}
\|(\D_l B_a) \, U\|_{HS}^2
& \leq & 2 \, \sup_{1 \leq l \leq d} \|\D_l^2 C\|_\infty \, 
	\sum_{k=1}^d (V^* \, R_s \, V)_{kk} \, |(D_P)_{kk}|^2
\\
& = & 2 \, \sup_{1 \leq l \leq d} \|\D_l^2 C\|_\infty \, \tr (V^* \, R_s \, V \, D_P^2)
	=2 \, \sup_{1 \leq l \leq d} \|\D_l^2 C\|_\infty \, \tr (R_s \, P^2)
\\*
& = & 2 \, \sup_{1 \leq l \leq d} \|\D_l^2 C\|_\infty \, \tr (U \, R_s \, \overline{U}).
\end{eqnarray*}
This completes the proof of the proposition under the assumptions (a) and (b).

Next we turn to (c).
Suppose (c) holds.
Write $\widetilde{C} = \widetilde{R} + i \, \widetilde{B}$.
Set $\widetilde{R}_s = \frac{1}{2} \, (\widetilde{R} + \widetilde{R}^T)$ and $\widetilde{B}_a = \frac{1}{2} \, (\widetilde{B} - \widetilde{B}^T)$.
Since $\phi$ is real-valued, we have $R_s = \phi \, \widetilde{R}_s$ and $B_a = \phi \, \widetilde{B}_a$.
Applying Lemma \ref{W f} to $\phi$ we obtain $(\D_l \phi)^2 \leq 2 \, \|\phi\|_{W^{2,\infty}} \, \phi$.
By Lemmas \ref{W Ba < Rs} and \ref{W Q xi square general} we also have $\|\widetilde{B}_a \, \xi\|^2 \leq 4 \, \|\widetilde{R}_s\|_\infty \, (\widetilde{R}_s \, \xi, \xi)$ for all $\xi \in \Ci^d$.
Therefore
\begin{eqnarray*}
\|(\D_l B_a) \, U\|_{HS}^2
& = & \sum_{j=1}^d \|(\D_l B_a) \, U \, e_j\|_2^2
	= (\D_l \phi)^2 \, \sum_{j=1}^d \|\widetilde{B}_a \, U \, e_j\|_2^2
\\
& \leq & 8 \, \|\phi\|_{W^{2,\infty}} \, \|\widetilde{R}_s\|_\infty \, \phi	\, \sum_{j=1}^d (\widetilde{R}_s \, U \, e_j, U \, e_j)
	= 8 \, \|\phi\|_{W^{2,\infty}} \, \|\widetilde{R}_s\|_\infty \, \tr(U \, R_s \, \overline{U}).
\end{eqnarray*}
The proof is complete.
\end{proof}

Our next aim is to show that if $D(A) \subset W^{1,2}(\Ri^d)$, then $C_c^\infty(\Ri^d)$ is a core for $A$.

\begin{lemm} \label{W A mollifier lemma}
Suppose $D(A) \subset W^{1,2}(\Ri^d)$.
Then
\[
\sum_{k,l=1}^d \int_{\Ri^d} c_{kl} \, \eta \, (\D_k u) \, \D_l \overline{\phi}
= \Big( \eta \, Au - \sum_{k,l=1}^d c_{kl} \, (\D_k u) \, \D_l \eta, \phi \Big)
\]
for all $u \in D(A)$ and $\eta, \phi \in C_c^\infty(\Ri^d)$.
\end{lemm}

\begin{proof}
Let $u \in D(A)$ and $\eta, \phi \in C_c^\infty(\Ri^d)$.
Then
\begin{eqnarray*}
(\eta \, Au, \phi)
& = & (Au, \overline{\eta} \, \phi)
	= \gota(u, \overline{\eta} \, \phi)
	= \sum_{k,l=1}^d \int_{\Ri^d} c_{kl} \, (\D_k u) \, \D_l (\eta \, \overline{\phi})
\\
& = & \sum_{k,l=1}^d \int_{\Ri^d} c_{kl} \, (\D_k u) \, (\D_l \eta) \, \overline{\phi}
	+ \sum_{k,l=1}^d \int_{\Ri^d} c_{kl} \, (\D_k u) \, \eta \, \D_l \overline{\phi}.
\end{eqnarray*}
Next we rearrange the terms to derive the lemma.
\end{proof}

Recall that $J_n$ is the usual mollifier with respect to a suitable function in $C_c^\infty(\Ri^d)$ for all $n \in \Ni$.

\begin{prop} \label{W A mollifier}
Suppose $D(A) \subset W^{1,2}(\Ri^d)$.
Then $C_c^\infty(\Ri^d)$ is a core for $A$ if and only if $\lim_{n \to \infty} A(J_n * u) = Au$ in $L_2(\Ri^d)$ for all $u \in D(A)$.
\end{prop}

\begin{proof}
($\Longrightarrow$) It is well-known that $\lim_{n \to \infty} J_n * (Au) = Au$ in $L_2(\Ri^d)$.
Therefore it suffices to show that $\lim_{n \to \infty} \|A(J_n * u) - J_n * (Au)\|_2 = 0$.

By a similar calculation as in \eqref{W mollifying difference} we yield 
\begin{equation} \label{W mollifying difference on A 1}
A(J_n * u) - J_n * Au = T_n u
\end{equation}
for all $n \in \Ni$ and $u \in C_c^\infty(\Ri^d)$, where the bounded operator $T_n: W^{1,2}(\Ri^d) \longrightarrow L_2(\Ri^d)$ is defined by \eqref{W Tn definition}.
Let $n \in \Ni$ and $u \in D(A)$.
Since $C_c^\infty(\Ri^d)$ is a core for $D(A)$, there exists a sequence $\{\phi_j\}_{j \in \Ni}$ in $C_c^\infty(\Ri^d)$ such that 
\begin{equation} \label{W phij to u in D(A)}
\lim_{j \to \infty} \phi_j = u
\end{equation}
in $D(A)$.
By hypothesis $D(A) \subset W^{1,2}(\Ri^d)$.
Therefore the inclusion $D(A) \hookrightarrow W^{1,2}(\Ri^d)$ is continuous.
It follows from \eqref{W phij to u in D(A)} that $\lim_{j \to \infty} \phi_j = u$ in $W^{1,2}(\Ri^d)$.
Recall that the operator $T_n$ is bounded.
As a consequence $\lim_{j \to \infty} T_n \phi_j = T_n u$ in $L_2(\Ri^d)$.
We also derive from \eqref{W phij to u in D(A)} that $\lim_{j \to \infty} J_n * \phi_j = J_n * u$ in $L_2(\Ri^d)$ and $\lim_{j \to \infty} J_n * (A \phi_j) = J_n * (Au)$ in $L_2(\Ri^d)$.
Therefore \eqref{W mollifying difference on A 1} gives
\[
\lim_{j \to \infty} A(J_n * \phi_j)
= \lim_{j \to \infty} \big( T_n \phi_j + J_n * (A \phi_j) \big)
= T_n u + J_n * (Au)
\]
in $L_2(\Ri^d)$.
Since $T_n$ is bounded, it is also closed.
Hence $J_n * u \in D(A)$ and $A(J_n * u) = T_n u + J_n * (Au)$.
That is, 
\begin{equation} \label{W mollifying difference on A 2}
A(J_n * u) - J_n * Au = T_n u
\end{equation}
also holds for all $n \in \Ni$ and $u \in D(A)$.

Let $\psi \in W^{2,2}(\Ri^d)$.
Then $\lim_{n \to \infty} J_n * \psi = \psi$ in $W^{2,2}(\Ri^d)$.
Consequently $\lim_{n \to \infty} A(J_n * \psi) = A \psi$ in $L_2(\Ri^d)$.
Also $\lim_{n \to \infty} J_n * (A \psi) = A \psi$ in $L_2(\Ri^d)$.
Therefore it follows from \eqref{W mollifying difference on A 2} that $\lim_{n \to \infty} \|T_n u\|_2 = 0$.
This is for all $\psi \in W^{2,2}(\Ri^d)$.
Since $W^{2,2}(\Ri^d)$ is dense in $W^{1,2}(\Ri^d)$ and $\{T_n\}_{n \in \Ni}$ is bounded by Lemma \ref{W Tn}, we deduce that $\lim_{n \to \infty} \|T_n u\|_2 = 0$ for all $u \in W^{1,2}(\Ri^d)$.
In particular $\lim_{n \to \infty} \|T_n u\|_2 = 0$ for all $u \in D(A)$ as $D(A) \subset W^{1,2}(\Ri^d)$ by hypothesis.

($\Longleftarrow$) Let $\tau \in C_c^\infty(\Ri^d)$ be such that $0 \leq \tau \leq \one$, $\tau|_{B_1(0)} = 1$ and $\supp \tau \subset B_2(0)$.
Define $\tau_n(x) = \tau(n^{-1} \, x)$ for all $x \in \Ri^d$ and $n \in \Ni$.

Let $n \in \Ni$.
Let $u \in D(A)$ and $\phi \in C_c^\infty(\Ri^d)$.
Then $u \in W^{1,2}(\Ri^d)$ and hence $\tau_n \, u \in W^{1,2}(\Ri^d)$.
Moreover
\[
\gota(\tau_n \, u, \phi)
= \sum_{k,l=1}^d \int_{\Ri^d} c_{kl} \, \D_k (\tau_n \, u) \, \D_l \overline{\phi}
= \sum_{k,l=1}^d \int_{\Ri^d} c_{kl} \, ((\D_k \tau_n) \, u + \tau_n \, \D_k u) \, \D_l \overline{\phi}
= (f_n, \phi),
\]
where 
\[
f_n 
= (A \tau_n) \, u + \tau_n \, Au 
	- \sum_{k,l=1}^d c_{kl} \, (\D_k \tau_n) \, \D_l u - \sum_{k,l=1}^d c_{kl} \, (\D_l \tau_n) \, \D_k u
\]
and we used Lemma \ref{W A mollifier lemma} in the last equality.
Since $f_n \in L_2(\Ri^d)$, we have $\tau_n \, u \in D(A)$ and $A(\tau_n \, u) = f_n$.
Next we will show that $\lim_{n \to \infty} f_n = Au$ in $L_2(\Ri^d)$.
Clearly $\lim_{n \to \infty} \tau_n \, Au = Au$ in $L_2(\Ri^d)$.
Note that
\begin{eqnarray*}
\|(A \tau_n) \, u\|_2
& = & \Big\| - \sum_{k,l=1}^d (\D_l (c_{kl} \, \D_k \tau_n)) \, u \Big\|_2
	= \Big\| \sum_{k,l=1}^d \Big( (\D_l c_{kl}) \, \D_k \tau_n + c_{kl} \, \D_l \D_k \tau_n \Big) \, u \Big\|_2
\\
& \leq & \sum_{k,l=1}^d \|c_{kl}\|_{W^{2,\infty}} \, \Big( \frac{1}{n} \, \|\D_k \tau\|_\infty 
	+ \frac{1}{n^2} \, \|\D_l \D_k \tau\|_\infty \Big) \, \|u\|_2.
\end{eqnarray*}
Similarly 
\[
\Big\| \sum_{k,l=1}^d c_{kl} \, (\D_k \tau_n) \, \D_l u \Big\|_2
\leq \frac{1}{n} \, \sum_{k,l=1}^d \|c_{kl}\|_\infty \, \|\D_k \tau\|_\infty \, \|\D_l u\|_2
\]
and
\[
\Big\| \sum_{k,l=1}^d c_{kl} \, (\D_l \tau_n) \, \D_k u \Big\|_2
\leq \frac{1}{n} \, \sum_{k,l=1}^d \|c_{kl}\|_\infty \, \|\D_l \tau\|_\infty \, \|\D_k u\|_2.
\]
It follows that these three terms go to 0 in $L_2(\Ri^d)$ as $n$ tends to infinity.
Hence 
\begin{equation} \label{W A limit 1}
\lim_{n \to \infty} \|A(\tau_n \, u) - Au\|_2 = 0.
\end{equation}

Finally we will show that $C_c^\infty(\Ri^d)$ is a core for $A$.
Let $u \in D(A)$.
The hypothesis gives 
\begin{equation} \label{W A limit 2}
\lim_{k \to \infty} \|A(J_k * (\tau_n \, u)) - A(\tau_n \, u)\|_2 = 0
\end{equation}
for all $n \in \Ni$.
Let $\varepsilon > 0$.
By \eqref{W A limit 1} we can choose an $n \in \Ni$ such that $\|A(\tau_n \, u) - Au\|_2 < \frac{\varepsilon}{2}$.
Next we use \eqref{W A limit 2} to choose a $k \in \Ni$ such that $\|A(J_k * (\tau_n \, u)) - A(\tau_n \, u)\|_2 < \frac{\varepsilon}{2}$.
Then 
\[
\|A(J_k * (\tau_n \, u)) - Au\|_2 
\leq \|A(J_k * (\tau_n \, u)) - A(\tau_n \, u)\|_2 + \|A(\tau_n \, u) - Au\|_2 
< \varepsilon.
\]
Note that $J_k * (\tau_n \, u) \in C_c^\infty(\Ri^d)$.
Hence $C_c^\infty(\Ri^d)$ is indeed a core for $A$.
\end{proof}

Let $\delta \in (0, 1)$.
Define
\[
C_\delta = (R_s + i \delta \, B_a) + i \, (B_s - i \, R_a).
\]

\begin{lemm}
The matrix $C_\delta$ takes values in $\Sigma_\psi$, where $\psi \in [0, \frac{\pi}{2})$ is such that $\tan \psi = \frac{1}{\delta} \, \tan\theta$.
\end{lemm}

\begin{proof}
Let $\xi \in \Ci^d$.
Then
\begin{eqnarray*}
|((\I C_\delta) \, \xi, \xi)|
& = & |((\I C) \, \xi, \xi)|
	\leq \tan\theta \, ((\R C) \, \xi, \xi)
	= \frac{1}{\delta} \, \tan\theta \, ((\delta \, R_s + i \delta \, B_a) \, \xi, \xi)
\\
& \leq & \frac{1}{\delta} \, \tan\theta \, ((R_s + i \delta \, B_a) \, \xi, \xi)
	= \frac{1}{\delta} \, \tan\theta \, ((\R C_\delta) \, \xi, \xi)
\end{eqnarray*}
since $C$ takes values in $\Sigma_\theta$ and $(R_s \, \xi, \xi) \geq 0$ by Lemma \ref{W Ba < Rs}.
The statement now follows.
\end{proof}

Define the form
\[
\gota_{0,\delta}(u, v) = \int_{\Ri^d} (C_\delta \, \nabla u, \nabla u)
\]
on the domain $D(\gota_{0,\delta}) = C_c^\infty(\Ri^d)$.
Then by the same analysis as in Section \ref{S4.1}, the form $\gota_{0,\delta}$ is closable.
Let $A_\delta$ be the operator associated with the closure of $\gota_{0,\delta}$.
Then we also have that $W^{2,2}(\Ri^d) \subset D(A_\delta)$ and 
\[
A_\delta u = - \sum_{k,l=1}^d \D_l ((C_\delta)_{kl} \, \D_k u)
\]
for all $u \in W^{2,2}(\Ri^d)$.
Define 
\[
H_\delta = - \sum_{k,l=1}^d \D_k (\overline{(C_\delta)_{kl}} \, \D_l u)
\]
on the domain $D(H_\delta) = C_c^\infty(\Ri^d)$.
Then we have the following.

\begin{prop} \label{W Adelta core lemma}
The space $C_c^\infty(\Ri^d)$ is a core for $A_\delta$.
Furthermore $A_\delta = (H_\delta)^*$.
\end{prop}

\begin{proof}
We note that 
\[
\tr (U \, (\R C_\delta) \, \overline{U})
= (1 - \delta) \, \tr (U \, R_s \, \overline{U}) 
	+ \delta \, \tr (U \, (\R C) \, \overline{U})
\geq (1 - \delta) \, \tr (U \, R_s \, \overline{U}) 
\]
for all $u \in C_c^\infty(\Ri^d)$, where $U = (\D_l \D_k u)_{1 \leq k,l \leq d}$.
That is, $C_\delta$ satisfies Condition \ref{C=phi C tilde} in Theorem \ref{W main theorem L2}.
Hence $C_c^\infty(\Ri^d)$ is a core for $A_\delta$ and $A_\delta = (H_\delta)^*$ by Theorem \ref{W main theorem L2}.
\end{proof}

\begin{lemm} \label{W A subset A delta}
Suppose $D(A) \subset W^{1,2}(\Ri^d)$.
Then $D(A) \subset D(A_\delta) \cap D((B_a)^\op)$ and
\[
A u = A_\delta u + i (1 - \delta) \, (B_a)^\op u
\]
for all $u \in D(A)$.
\end{lemm}

\begin{proof}
Recall that the operators $H_2$ and $L$ are defined by \eqref{W Hq} and \eqref{W L} respectively.
First note that $D(A) \subset W^{1,2}(\Ri^d) \subset D((B_a)^\op)$.
Moreover, the condition $D(A) \subset W^{1,2}(\Ri^d)$ implies that 
\[
(u, H_2 \phi)
= - \int_{\Ri^d} u \, \D_k (c_{kl} \, \D_l \overline{\phi})
= \int_{\Ri^d} c_{kl} \, (\D_k u) \, \D_l \overline{\phi}
= \gota(u, \phi)
= (Au, \phi)
\]
for all $u \in D(A)$ and $\phi \in C_c^\infty(\Ri^d)$, where we used integration by parts in the second step.
Since $Au \in L_2(\Ri^d)$, we conclude that $u \in D(B_2)$ and 
\begin{equation} \label{S4.6 A in B2}
B_2 u = Au
\end{equation}
for all $u \in D(A)$.
Therefore we also have $D(A) \subset D(B_2)$.

Next let $u \in D(A)$.
Then
\begin{eqnarray*}
(u, H_\delta \phi)
& = & (u, H_2 \phi) - i (1 - \delta) \, (u, L \phi)
	= (B_2 u, \phi) - i (1 - \delta) \, \big( (B_a)^\op u, \phi \big)
\\
& = & \big( B_2 u - i (1 - \delta) \, (B_a)^\op u, \phi \big)
\end{eqnarray*}
for all $\phi \in C_c^\infty(\Ri^d)$.
Note that $B_2 u - i (1 - \delta) \, (B_a)^\op u \in L_2(\Ri^d)$.
Hence $u \in D(A_\delta)$ and 
\[
A_\delta u = B_2 u - i (1 - \delta) \, (B_a)^\op u = A u - i (1 - \delta) \, (B_a)^\op u,
\]
where we used \eqref{S4.6 A in B2} in the last step.
The lemma now follows.
\end{proof}

\begin{lemm} \label{W A delta replacement condition}
Suppose $D(A) \subset W^{1,2}(\Ri^d)$.
Then there exists a $\delta_0 \in (0, 1)$ such that for all $\delta \in [\delta_0, 1)$ there exists an $M > 0$ such that $D(A_\delta) \subset W^{1,2}(\Ri^d)$ and $\|u\|_{W^{1,2}} \leq M \, \|u\|_{D(A_\delta)}$ for all $u \in D(A_\delta)$.
\end{lemm}

\begin{proof}
Since $D(A) \subset W^{1,2}(\Ri^d)$, there exists an $M_1 > 0$ such that $\|u\|_{W^{1,2}} \leq M_1 \, \|u\|_{D(A)}$ for all $u \in D(A)$ by the closed graph theorem.
Similarly the inclusion $W^{1,2}(\Ri^d) \subset D((B_a)^\op)$ implies that there exists an $M_2 > 0$ which satisfies $\|u\|_{D((B_a)^\op)} \leq M_2 \, \|u\|_{W^{1,2}}$ for all $u \in D((B_a)^\op)$.
Let $\delta_0 = (1 - \frac{1}{2 M_1 M_2}) \vee \frac{1}{2}$ and $\delta \in [\delta_0, 1)$.
If $u \in D(A)$ then $u \in D(A_\delta)$ by Lemma~\ref{W A subset A delta}.
Therefore
\pagebreak[1]
\begin{eqnarray*}
\|u\|_{W^{1,2}}
& \leq & M_1 \, (\|u\|_2 + \|Au\|_2)
\leq M_1 \, (\|u\|_2 + \|A_\delta u\|_2 + (1-\delta) \, \|(B_a)^\op u\|_2)
\\
& = & M_1 \, \|u\|_{D(A_\delta)} + (1-\delta) \, M_1 \|(B_a)^\op u\|_2
\leq M_1 \, \|u\|_{D(A_\delta)} + (1-\delta) \, M_1 \, M_2 \, \|u\|_{W^{1,2}}
\end{eqnarray*}
for all $u \in D(A)$.
It follows that 
\[
\|u\|_{W^{1,2}} \leq \frac{M_1}{1 - (1-\delta) \, M_1 M_2} \, \|u\|_{D(A_\delta)}
\]
for all $u \in D(A)$.
In particular 
\begin{equation} \label{W extension to domain Adelta}
\|u\|_{W^{1,2}} \leq \frac{M_1}{1 - (1-\delta) \, M_1 M_2} \, \|u\|_{D(A_\delta)}
\end{equation}
for all $u \in C_c^\infty(\Ri^d)$.
Note that $C_c^\infty(\Ri^d)$ is a core for $A_\delta$ by Lemma \ref{W Adelta core lemma} and the space $W^{1,2}(\Ri^d)$ is complete.
Consequently \eqref{W extension to domain Adelta} implies that $D(A_\delta) \subset W^{1,2}(\Ri^d)$ and 
\[
\|u\|_{W^{1,2}} \leq \frac{M_1}{1 - (1-\delta) \, M_1 M_2} \, \|u\|_{D(A_\delta)}.
\]
for all $u \in D(A_\delta)$ as required.
\end{proof}

\begin{lemm} \label{W A delta mollifier}
Let $u \in D(A)$.
Then $\lim_{n \to \infty} A_\delta(J_n * u) = A_\delta u$ in $L_2(\Ri^d)$.
\end{lemm}

\begin{proof}
The proof is the same as that of the `only if' part of Proposition \ref{W A mollifier}.
Note that $C_c^\infty(\Ri^d)$ is a core for $A_\delta$ by Lemma \ref{W Adelta core lemma} and $D(A_\delta) \subset W^{1,2}(\Ri^d)$ by Lemma \ref{W A delta replacement condition}.
\end{proof}

We are now in the position to prove Theorem \ref{W smoothness to core}.

\begin{proof}[Proof of Theorem \ref{W smoothness to core}]
Let $\delta = \delta_0$, where $\delta_0$ is defined as in Lemma \ref{W A delta replacement condition}.
By Lemma \ref{W A delta mollifier} we have $\lim_{n \to \infty} A_\delta(J_n * u) = A_\delta u$ in $L_2(\Ri^d)$ for all $u \in D(A)$.
Furthermore \cite[Proposition 2.1]{ERS5} gives that $\lim_{n \to \infty} (B_a)^\op (J_n * u) = (B_a)^\op u$ in $L_2(\Ri^d)$ for all $u \in D((B_a)^\op)$.
Hence $\lim_{n \to \infty} A(J_n * u) = Au$ in $L_2(\Ri^d)$ for all $u \in D(A)$ as $A \subset A_\delta + i (1 - \delta) \, (B_a)^\op$.
Using Proposition \ref{W A mollifier} we can conclude that $C_c^\infty(\Ri^d)$ is a core for $A$.
\end{proof}

\section{Examples}

In this section we present several applications of Theorems \ref{main theorem higher dimensions}, \ref{W smoothness to core} and \ref{W main theorem L2} in showing the core properties for some specific degenerate elliptic operators in higher dimensions.

\begin{eg}
For all $(x,y) \in \Ri^2$ let $\phi(x,y) = \frac{\pi}{4} \, \cos (\sin (x+y))$.
Let 
\[
C 
= \left( 
\begin{array}{cc}
2 \, \cos \phi + i \, \sin \phi & \sin \phi 
\\
- \sin \phi & 2 \, \cos \phi + i \, \sin \phi
\end{array}
\right)
.
\]
Then $\big( C(x,y) \, \xi, \xi \big) \in \Sigma_{\frac{\pi}{4}}$ for all $(x,y) \in \Ri^2$ and $\xi \in \Ci^2$.
Note that $B_a = 0$.

Consider the form $\gota_0$ defined by
\[
\gota_0(u,v) = \int_{\Ri^2} (C \, \nabla u, \nabla v)
\]
on the domain $D(\gota_0) = C_c^\infty(\Ri^2)$.
Then $\gota_0$ is closable.
Let $A$ be the operator associated with the closure of $\gota_0$ in $L_2(\Ri^2)$.
Since $B_a = 0$, we can extend the contraction $C_0$-semigroup $S$ generated by $-A$ to a contraction $C_0$-semigroup $S^{(p)}$ on $L_p(\Ri^2)$ for all $p \in [4-2 \sqrt{2}, 4+2 \sqrt{2}]$ by Proposition \ref{W Lp extension}.
Let $-A_p$ be the generator of $S^{(p)}$ for all $p \in [4-2 \sqrt{2}, 4+2 \sqrt{2}]$.
Then the space $C_c^\infty(\Ri^2)$ is a core for $A_p$ for all $p \in (4-2 \sqrt{2}, 4+2 \sqrt{2})$ by Theorem \ref{main theorem higher dimensions}.
\end{eg}

\begin{eg}
For all $(x,y) \in \Ri^2$ let
\[
C(x,y)
= \left( 
\begin{array}{cc}
\frac{1}{\sqrt{2}}(1+i) & e^{i \, (x+y)}
\\
i \, e^{-i \, (x+y)} & \frac{1}{\sqrt{2}}(1+i)
\end{array}
\right)
.
\]
Note that 
\begin{equation} \label{W example 2}
C = (1 + i) \, (\R C),
\end{equation}
where 
\[
(\R C)(x,y)
= \left( 
\begin{array}{cc}
\frac{1}{\sqrt{2}} & \frac{\cos (x+y) + \sin (x+y)}{2} - i \frac{\cos (x+y) - \sin (x+y)}{2} 
\\
\frac{\cos (x+y) + \sin (x+y)}{2} + i \frac{\cos (x+y) - \sin (x+y)}{2} & \frac{1}{\sqrt{2}}
\end{array}
\right)
.
\]
Therefore $\big( C(x,y) \, \xi, \xi \big) \in \Sigma_{\frac{\pi}{4}}$ for all $(x,y) \in \Ri^2$ and $\xi \in \Ci^2$.

Consider the form $\gota_0$ defined by
\[
\gota_0(u,v) = \int_{\Ri^2} (C \, \nabla u, \nabla v)
\]
on the domain $D(\gota_0) = C_c^\infty(\Ri^2)$.
Then $\gota_0$ is closable.
Let $A$ be the operator associated with the closure of $\gota_0$ in $L_2(\Ri^2)$.

Using \eqref{W example 2} and the fact that $\R C$ is self-adjoint, we conclude that the space $C_c^\infty(\Ri^2)$ is a core for $A$ by Theorem \ref{W main theorem L2}(i).
\end{eg}

\begin{eg}
Let $c_{kl} \in \Ci$ for all $k,l \in \{1, 2\}$.
Suppose there exists a constant $\mu > 0$ such that
\[
\R (C \, \xi, \xi) \geq \mu \, \|\xi\|^2
\]
for all $\xi \in \Ci^2$, where $C = (c_{kl})_{1 \leq k,l \leq 2}$.
Define $A_1 = \D_x$ and $A_2 = \cos x \, \D_y + \sin x \, \D_z$.
Consider the form $\gota_0$ defined by
\[
\gota_0(u,v) = \sum_{k,l=1}^2 \int_{\Ri^3} c_{kl} (A_k u) \, A_l v
\]
on the domain $D(\gota_0) = C_c^\infty(\Ri^3)$.
Then $\gota_0$ is closable.
Let $A$ be the operator associated with the closure of $\gota_0$ in $L_2(\Ri^3)$.
Then formally
\[
A = - \sum_{k,l=1}^2 c_{kl} A_l \, A_k.
\]
We have $D(A) \subset W^{1,2}(\Ri^3)$.
This follows from the regularity of sub-elliptic operators on Lie groups associated to unitary representations.
Specifically it follows from \cite[Theorem 9.2.II]{ER13} together with \cite[Lemma 6.1]{ER1} and \cite[Theorem 7.2.(VI and V)]{ER1} applied to the standard representation of the covering group of the Euclidean motion group (cf.\ \cite[Example II.5.1]{DER4}).

Hence $C_c^\infty(\Ri^3)$ is a core for $A$ by Theorem \ref{W smoothness to core}.
\end{eg}

\subsection*{Acknowledgements}
I wish to thank Tom ter Elst and Boris Baeumer for giving detailed and valuable comments.


\end{document}